\documentclass[a4paper,11pt]{article}
\title{Tensor train methods for high-dimensional nonlinear filtering problems with correlated noise}
\author{Yuhua Meng\thanks{Department of Mathematics, The University of Hong Kong, Pokfulam Road, Hong Kong SAR, P.R.China. Email: {yvette00@connect.hku.hk.} } 
\and Stephen S.-T. Yau\thanks{Corresponding author. Department of Mathematical Sciences, Tsinghua University, Beijing 100084, P. R. China. Email: {yau@uic.edu.} }
\and Zhiwen Zhang\thanks{Corresponding author. Department of Mathematics, The University of Hong Kong, Pokfulam Road, Hong Kong SAR, P.R. China, and the Materials Innovation Institute for Life Sciences and Energy (MILES), HKU-SIRI, Shenzhen, 518045, P.R. China. Email: {zhangzw@hku.hk.}}
}

\usepackage[margin=1in]{geometry}
\usepackage{graphicx}
\usepackage{multirow}
\usepackage{subcaption}  
\usepackage{xcolor}
\usepackage[shortlabels]{enumitem}
\usepackage{amsmath,amsfonts,amsthm,amssymb}
\usepackage{cleveref}
\usepackage{xcolor}
\usepackage{verbatim}
\usepackage{placeins}
\usepackage{bm}
 \usepackage{booktabs} 
 \usepackage{algorithm}
 \usepackage{algpseudocode}
 \usepackage[backend=biber, style=numeric, maxnames=99]{biblatex}
\newtheorem{theorem}{Theorem}[section]
\newtheorem{assumption}[theorem]{Assumption}

\newtheorem{lemma}[theorem]{Lemma}

\newtheorem{remark}[theorem]{Remark}
\allowdisplaybreaks
\algnewcommand{\Input}{\item[\textbf{Input:}]}
\algnewcommand{\Output}{\item[\textbf{Output:}]}
\begin{document}
%\title{Regularity Estimate of pathwise robust DMZ equation and its sparse approximation}
\maketitle

% REQUIRED
\begin{abstract}
Nonlinear filtering with correlated noise leads to a Duncan-Mortensen-Zakai (DMZ) equation in the form of a stochastic partial differential equation (SPDE). Unlike the independent noise case, the presence of correlation prevents the classical invertible transformation that reduces the DMZ equation to a deterministic partial differential equation, requiring a direct numerical treatment of the SPDE. This paper develops a tensor train (TT) based framework for solving medium- to high-dimensional DMZ equations with correlated noise. Spatial discretization transforms the SPDE into a high-dimensional stochastic differential system, which is efficiently compressed using TT approximation. A semi-implicit Milstein scheme is employed for temporal integration to ensure stability and accuracy. Under suitable regularity assumptions, we establish a convergence analysis of the proposed method. In particular, the spatial error is controlled by both the mesh size and the prescribed TT approximation accuracy. In the temporal direction, the convergence is proved by estimating stochastic integrals involving drifted observations, without invoking a change-of-measure argument. Numerical experiments demonstrate that the proposed method achieves stable and accurate performance for cubic sensor problems. In challenging multi-modal settings, where particle filter and extended Kalman filter deteriorate, the proposed method maintains accuracy and effectively captures the posterior distribution.


%\noindent \textbf{MSC:} 

\noindent
\end{abstract}

\section{Introduction}
	\label{sec:introduction}
The nonlinear filtering (NLF) problem concerns the estimation of an unobservable stochastic state process $x_t$ from noisy observations $\{y_s:0 \le s \le t\}$. Its study dates back to the early 1960s, when the Kalman filter \cite{kalman1960new} and the Kalman--Bucy filter \cite{kalman1961new} were introduced, providing optimal estimation for linear systems under Gaussian assumptions. For nonlinear dynamics, extensions such as the extended Kalman filter (EKF) and the unscented Kalman filter have been developed, but they remain restricted by Gaussian assumptions. Particle filters (PF) offer greater flexibility, as they can handle nonlinear and non-Gaussian problems. However, PFs suffer from well-known issues such as weight degeneracy and sample impoverishment, and their computational cost increases rapidly with the state dimension, limiting their effectiveness in high-dimensional settings. An alternative formulation of the NLF problem is based on stochastic partial differential equations (SPDEs). In particular, the conditional probability density of $x_t$ given the observation history $\sigma\{y_s: 0 \le s \le t\}$ satisfies the Kushner-Stratonovich equation \cite{kushner1967dynamical,stratonovich1959optimum}, which is a nonlinear and nonlocal SPDE. By introducing the Radon-Nikodym derivative, this equation can be transformed into a linear SPDE governing the unnormalized conditional density, known as the Duncan-Mortensen-Zakai (DMZ) equation \cite{duncan1967probability,mortensen1966optimal,zakai1969optimal}. Various numerical methods have been proposed for solving the DMZ equation, including the splitting-up algorithm \cite{bensoussan1990approximation}, and the $S^3$ algorithm \cite{lototsky1997nonlinear}. Alternatively, the Yau-Yau algorithm transforms the DMZ equation into a deterministic partial differential equation (PDE), enabling an efficient offline-online decomposition \cite{yau2000real,yau2008real,luo2013complete,wang2019proper,li2022solving}.

The methods mentioned above are developed under the assumption that the Brownian motions driving $x_t$ and $y_t$ are mutually independent. When the noises are correlated, the KF, PF, and their variants can be readily generalized \cite{jazwinski2007stochastic,luo2019feedback,kang2022optimal}, but they suffer from the same limitations discussed earlier. For the SPDE-based approach, available results remain limited. Extensions of the splitting-up algorithm to correlated noise have been established in \cite{florchinger1991time,luo2023convergence}, where numerical experiments are restricted to low-dimensional settings. Hermite-Galerkin methods demonstrate their effectiveness in 2D NLF problems \cite{sun2023solving}. Reference \cite{lototsky2003nonlinear} solves a 5D NLF problem in the linear case. Solving NLF problems with correlated noise remains challenging in medium- to high-dimensional settings.

A fundamental difficulty in solving the DMZ equation numerically is the rapid growth of storage and computation cost with respect to the dimension. To address this issue, we employ the tensor train (TT) decomposition, which provides an efficient low-rank representation for high-dimensional tensors. By compressing the tensor to a low-rank structure, TT has been applied to a range of medium- to high-dimensional PDEs, including the Fokker-Planck equation \cite{dolgov2012fast}, stochastic elliptic equations \cite{dolgov2015polynomial}, parabolic PDEs \cite{richter2024continuous}, and Helmholtz equations \cite{feng2025functional}. In the context of nonlinear filtering, TT-based methods have also been developed for the DMZ equation. In particular, a TT-based finite difference method (FDM) was proposed for the independent noise case in \cite{li2022solving}, and further improved under a functional polyadic assumption in \cite{meng2025regularity}. Besides the TT method, several other approaches are available to mitigate the curse of dimensionality. For example, in spectral methods, hyperbolic cross approximation \cite{luo2013hermite} and scaling or translating factors \cite{sun2023solving} have been introduced to reduce complexity.

This paper extends the application of TT-based FDM to the DMZ equation with correlated noise. The main contributions of this paper are summarized as follows:

\begin{itemize}
	\item We develop a TT-based semi-implicit Milstein method for solving the DMZ equation with correlated noise, where the classical transformation to a deterministic PDE is no longer applicable.
	
	\item We rigorously establish convergence of the proposed method in both space and time. In particular, the temporal convergence is proved by directly estimating stochastic integrals involving drifted observations, without invoking a change-of-measure argument.
	
	\item Numerical experiments demonstrate that the TT method provides a more favorable accuracy-cost tradeoff in the tested high-dimensional nonlinear examples.
\end{itemize}

To solve the DMZ equation, we first discretize it directly, which leads to a high-dimensional SDE. To render this system computationally tractable, we approximate the semi-discrete system in TT format, so that all the high-dimensional full arrays are represented in compressed form. Under suitable assumptions on the spatial regularity of the exact solution and certain conditions on the signal process model \eqref{eq:signal_model}, we establish convergence of the semi-discrete TT scheme with respect to the mesh size and the prescribed TT approximation accuracy. 

A variety of numerical schemes may be employed for SDEs; see \cite{2005Numerical} and references therein. To balance stability and efficiency, we adopt a semi-implicit Milstein scheme for temporal discretization. The temporal convergence is established by deriving suitable estimates for the related stochastic multiple integrals and applying a general convergence theorem from \cite{mil1988theorem}. In contrast to existing convergence analysis for the DMZ equation, which often relies on a change of measure and subsequent transformation back to the original measure via H{\" o}lder inequality \cite{luo2023convergence,dong2025brownian,newton1984discrete}, our analysis is carried out directly under the original probability measure. This avoids the loss of integrability inherent in measure-change arguments and leads to a more direct convergence analysis for SDEs driven by drifted observation processes.

To assess the performance of the proposed TT-based method, numerical experiments are presented on NLF problems with nonlinear dynamics and multi-modal problems. The approach is compared with PF and EKF in representative test cases. The results demonstrate that, by directly approximating the conditional probability density function and exploiting the TT format, the proposed method exhibits improved stability and accuracy in nonlinear regimes and successfully captures multi-modal features that are challenging for classical techniques.

The rest of the paper is organized as follows. Section \ref{sec:preliminaries} presents the DMZ equation, TT approximation, and the notation used in the time discretization scheme. The tensor train method is introduced in Section \ref{sec:TT method}. The convergence in space and time is proved in Sections \ref{sec:spatial convergence} and \ref{sec:time convergence}. In Section \ref{sec:numerical result}, we provide some numerical results. 

\section{Preliminaries}
\label{sec:preliminaries}
\subsection{DMZ equation}\label{subsec:DMZ}
The signal process with correlated noise is modeled by
\begin{align}\label{eq:signal_model}
	\begin{cases}
		dx_t=f(x_t)dt+G(x_t)dv_t + \rho(x_t)dw_t,\\
		dy_t=h(x_t)dt + dw_t,
	\end{cases}
\end{align}
where $x_t\in\mathbb{R}^d$ is the system state, $y_t\in\mathbb{R}^d$ is the observation at time $t$, $f$ and $h$ are $\mathbb{R}^d$-valued functions, $G$ and $\rho$ are matrix-valued functions, and $v_t$ and $w_t$ are standard Brownian motions. The unnormalized conditional density of $x_t$ given the observation history $\{y_s:0\le s\le t\}$ is governed by the DMZ equation \cite{duncan1967probability,mortensen1966optimal,zakai1969optimal}:
\begin{align}\label{eq:zakai}
	du(t,x)=\mathcal{L}_0udt+\sum_{k=1}^{d}\mathcal{L}_kudy_t^k,
\end{align}
where 
\begin{align*}
	\mathcal{L}_0(\cdot)&=\frac{1}{2}\sum_{ij}\frac{\partial^2\big( (GG^\top+\rho \rho^\top)_{ij}\cdot \big)}{\partial x_i\partial x_j} - \sum_{k=1}^d \frac{\partial}{\partial x_k}(f_k\cdot),\\
	\mathcal{L}_k(\cdot)&=h_k\cdot-\sum_{i=1}^d \frac{\partial}{\partial x_i}(\rho_{ik}\cdot).
\end{align*}
The DMZ equation is a stochastic partial differential equation (SPDE) driven by $y_t$. If $\rho(x)\equiv0$, then the state process and the observation process in \eqref{eq:signal_model} are independent. In this case, \eqref{eq:zakai} can be transformed into a deterministic PDE via an invertible exponential transformation \cite{davis1980multiplicative,yau2000real}. For the correlated noise setting considered here, no analogous transformation is available. Therefore, the SPDE \eqref{eq:zakai} needs to be treated directly.

\subsection{Tensor Train decomposition}

Finite difference methods are widely used to solve differential equations, but they suffer from the curse of dimensionality. In particular, the discretization of a $d$-dimensional PDE leads to a $d$-dimensional full grid tensor of size $N_1\times \cdots\times N_d$. This results in exponential growth in memory requirements and the number of arithmetic operations with respect to $d$. To enable an efficient numerical solution of higher-dimensional problems, it is necessary to exploit low-rank tensor representations, such as the Tensor Train (TT) decomposition.

The TT format was proposed as a stable approximation for a high-dimensional full tensor \cite{oseledets2011tensor}. Given a full grid tensor $ {\bf A}\in \mathbb{R}^{N_1\times\cdots\times N_d}$, its unfolding matrices are defined by 
\begin{align*}
	A_k=A_k(i_1,\dots,i_k;i_{k+1},\dots,i_d)={\bf A}(i_1,\dots,i_d).
\end{align*}
It was proved in \cite{oseledets2011tensor} that ${\bf A}$ admits a decomposition of the form
\begin{align}\label{eq:TT-decomp}
	A(i_1,\dots,i_d) = G_1(i_1)G_2(i_2)\cdots G_d(i_d),
\end{align}
where $r_0=r_d=1$, and $G_k(i_k)$ is an $r_{k-1}\times r_k$ matrix with $r_k\le \operatorname{rank}(A_k)$. A tensor $\bf A$ is said to be in TT format if it is written in \eqref{eq:TT-decomp}; $r_k$ is referred to as the TT-rank, and $G_k$ as the cores of the TT decomposition. Furthermore, if the unfolding matrices $A_k$ have low-rank approximations, i.e.,
\begin{align*}
	A_k=R_k+E_k,\quad ||E_k||_F=\varepsilon_k,\quad k=1,...,d-1,
\end{align*}
then the TT-SVD algorithm \cite{oseledets2011tensor} gives a TT approximation ${\bf B}$ with TT-rank $\operatorname{rank}(R_k)$ and an error bound $||{\bf A}-{\bf B}||_F\le \sqrt{\sum_{k=1}^{d-1}\varepsilon _k^2}$. The number of parameters in the TT format \eqref{eq:TT-decomp} is $\mathcal{O}(Ndr^2)$, where $N$ and $r$ are upper bounds on the mode size and TT-rank, respectively. Accordingly, the computational cost of basic linear algebra operations grows only polynomially with respect to $d$ \cite{oseledets2011tensor}. However, constructing the TT approximation from a full array requires $\mathcal{O}(N^dr^2)$ operations, which remains exponential in $d$. To tackle this problem, various fast approximation algorithms have been proposed \cite{savostyanov2011fast,savostyanov2014quasioptimality}. For a multi-dimensional matrix ${\bf M}$ of size $(n_1\times\cdots\times n_d)\times(m_1\times\cdots\times m_d)$, the corresponding TT format is also available \cite{oseledets2011tensormatrix,kazeev2012low,oseledets2010approximation}.  

\subsection{It\^o–Taylor Expansion and Notation}
To derive and analyze the time-discrete approximation of the semi-discrete scheme \eqref{eq:semi-discrete scheme}, we recall the notation used in the stochastic Taylor expansion, following standard references such as \cite{2005Numerical}. Specifically, the framework will be used to estimate multiple stochastic integrals and their convergence rates.

\medskip
\noindent\textbf{Time grid and filtration.} Let $0=\tau_0 <\tau_1<\cdots<\tau_{N_T}=T$ be a uniform time partition with step size $\delta=\frac{T}{N_T}$. Let $( \mathcal{F}_t )_{t\ge0}$ be a non-decreasing filtration, such that $x_t$, $y_t$ and ${\bf U}(t)$ are $\mathcal{F}_t$-adapted. Denote $\hat{\bf U}_n$ $(0\le n\le N_T)$ as the time-discrete approximation of the semi-discrete scheme \eqref{eq:semi-discrete scheme}. Each $\hat{\bf U}_n$ is $\mathcal{F}_{\tau_n}$-measurable, and the initial value satisfies $\hat{\bf U}_0={\bf U}(0)$. For brevity, we write $\mathbb{E}_{\tau_n}:=\mathbb{E}(\cdot|\mathcal{F}_{\tau_n})$.

\medskip
\noindent\textbf{Multi-index notation.}
For a multi-index $\alpha=(j_1,...,j_l)$ with $j_i=0,1,...,d$, define its length $l(\alpha)=l$, and let $\nu$ denote the empty index with $l(\nu)=0$. Let $n(\alpha)$ denote the number of zero components in $\alpha$. We use the standard notations $ -\alpha:=(j_2,...,j_l),\, \alpha-:=(j_1,...,j_{l-1})$, with the convention $-(j)=(j)-=\nu$.  Define the set of multi-indices 
$$ \mathcal{M}:= \{ (j_1,...,j_l)| j_i = 0,1,...,d, \,  l=1,2,...  \} \cup \{\nu \}.$$

\noindent For $\gamma\in\{0.5,1,1.5,\ldots\}$, define
$$ \mathcal{A}_\gamma :=\{ \alpha \in \mathcal{M}\,|\, l(\alpha) + n(\alpha)\le 2\gamma \text{ or } l(\alpha)=n(\alpha) = \gamma + \frac{1}{2} \},$$
and its remainder set 
$$ \mathcal{B}(\mathcal{A}_{\gamma}) :=\{ \alpha\in\mathcal{M}\setminus \mathcal{A}_{\gamma} : -\alpha\in\mathcal{A}_{\gamma} \}.$$

\medskip
\noindent\textbf{Stochastic multiple integrals.} For $\alpha\in\mathcal{M}$, define the differential operators recursively
\begin{align*}
	{\bf L}^{\nu}={\bf I},\qquad
	{\bf L}^{\alpha}= {\bf L}^{-\alpha}{\bf L}_{j_1},\,\alpha = (j_1,...,j_l),
\end{align*}
where ${\bf L}_j$ is defined in \eqref{eq:discretized_operator}. To define the stochastic multiple integrals, denote $\mathcal{H}_v$ as the set of all the adapted right-continuous stochastic processes $g=\{g(t),t\ge 0\}$ with $|g(t)|< \infty$ almost surely for $\forall t>0$; $\mathcal{H}_{(0)}$ as all those with $\int_0^t |g(s)|ds<\infty, \, \forall t>0$ almost surely; and $\mathcal{H}_{(1)}$ as all those with $\int_0^t|g(s)|^2ds,\,\forall t>0$ with probability 1. Denote $\mathcal{H}_{(j)}=\mathcal{H}_{(1)}$. For $\alpha=(j_1,...,j_l)\in \mathcal{M}$, define the stochastic multiple integrals over $[\tau,t]$:
\begin{align}\label{eq:def_stochastic integral}
	I_{\alpha}[g(\cdot)]_{\tau ,t}&=g(t),\quad &&\text{if } \alpha=\nu, \notag\\
	I_{\alpha}[g(\cdot)]_{\tau,t}&=\int_\tau^t I_{\alpha-}[g(\cdot)]_{\tau,s}ds,\quad &&\text{if } j_l=0,\notag\\
	I_{\alpha}[g(\cdot)]_{\tau,t}&=\int_\tau^t I_{\alpha-}[g(\cdot)]_{\tau,s}dy_s^{j_l},\quad &&\text{if } 1\le j_l\le d,
\end{align}
where $\mathcal{H}_\alpha$ is the set of adapted right-continuous stochastic process $g$ with $I_{\alpha-}[g]_{\tau,t}\in \mathcal{H}_{(j_l)}$. For $g\equiv1$, we write $I_{\alpha;\tau,t}$ as shorthand, and set $I_{\alpha}^n=I_{\alpha;\tau_n,\tau_{n+1}}$.

\medskip
\noindent\textbf{It\^o-Taylor expansion.} The solution ${\bf U}(t)$ of \eqref{eq:semi-discrete scheme} admits the representation
\begin{align}\label{eq:Ito-Taylor expansion}
	{\bf U}(t) = \sum_{\alpha\in\mathcal{A}_{\gamma}}{\bf L}^{\alpha}{\bf U}(\tau) I_{\alpha;\tau,t}+\sum_{\alpha\in\mathcal{B}({\mathcal{A}}_\gamma)} {\bf L}^{\alpha}I_\alpha[{\bf U}(\cdot)]_{\tau,t}\,. 
\end{align}
which forms the basis for time discretization and error analysis in the sequel.

	\section{TT-based discretization of the DMZ equation}\label{sec:TT method}
Throughout the paper, let ${\bf Id}^{(N)}_d$ denote the $d$-dimensional identity tensor of mode size $N$.

\medskip
\begin{assumption}\label{ass:matrix-simplify}
	Assume that $G(x)G(x)^\top$ and $\rho(x)\rho(x)^\top$ are diagonal matrices with nonnegative diagonal entries. Moreover, there exist constants $\lambda_G,\Lambda_G,\Lambda_\rho>0$ such that
	\(
	\lambda_G |\xi|^2
	\le
	\xi^\top G(x)G(x)^\top \xi
	\le
	\Lambda_G |\xi|^2,
	\)
	and
	\(
	\xi^\top \rho(x)\rho(x)^\top \xi
	\le
	\Lambda_\rho |\xi|^2,
	\)
	for all $x\in\mathbb R^d$ and $\xi\in\mathbb R^d$.
\end{assumption}

\begin{assumption}\label{ass:matrix-regularity}
	Assume that $\rho_{ij}(x)$ and $G_{ij}(x)$ belong to $C_b^4$, i.e., they are four times continuously differentiable and all derivatives up to order four are bounded.
\end{assumption}

\begin{assumption}\label{ass:regularity}
	$f(x):\mathbb{R}^d\rightarrow\mathbb{R}^d$ and $h(x):\mathbb{R}^d\rightarrow\mathbb{R}^d$ are $C^1$ with bounded derivatives, hence Lipschitz continuous with constants $L_f$, $L_h$, i.e., 
	$|f(x)-f(y)|\le L_f|x-y|$, $|h(x)-h(y)|\le L_h|x-y|$.
\end{assumption}

\begin{assumption}\label{ass:growth}
	The drift function satisfies linear growth bound: $|f(x)|^2\le K(1+|x|^2)$. The observation function $h$ is uniformly bounded: $|h(x)|<C_h$. 
\end{assumption}

Assumption \ref{ass:matrix-simplify} is introduced for simplicity. In fact, the discretization method can be readily generalized to the case where $G$, $\rho$ are general matrices. The corresponding difference method is outlined in the Appendix.

The assumptions are set for technical proof; however, the proposed method may still perform well even when they are not satisfied, as demonstrated in the numerical experiments in Subsection \ref{subsec: cubic sensor}.

\medskip
\noindent\textbf{Spatial discretization.} To apply the discretization, we constrain the domain to a cube ${\bf I}=[-L,L]^d$, and define the functions $x_{i,r}:=-L+r\cdot {\Delta x}$ with step $\Delta x=\frac{2L}{N+1}$, $r\in\frac{1}{2}\mathbb{Z}$. 

Under Assumption \ref{ass:matrix-simplify}, we have the semi-discrete scheme:
\begin{equation}\label{eq:semi-discrete scheme}
	d {\bf U}(t)= {\bf L}_0{\bf U}(t)dt+\sum_{k=1}^d{\bf L}_k{\bf U}(t)dy_t^k,
\end{equation}
where ${\bf L}_j$ is the discretized operator of $\mathcal{L}_j$, $0\le j\le d$.  Specifically, in one-dimensional space, define
\[
\overline{\bf C}_1 = \frac{1}{\Delta x}
\begin{pmatrix}
	1 & 0 & 0 & \cdots & 0 \\
	-1 & 1 & 0 & \cdots & 0 \\
	0 & -1 & 1 & \cdots & 0 \\
	0 & 0 & -1 & \ddots & 0 \\
	\vdots & \vdots & \ddots & \ddots & 1 \\
	0 & 0 & \cdots & 0 & -1 
\end{pmatrix}_{(N+1)\times N},
\]
\[{\bf C}_1 = \frac{1}{2\Delta x}\,\operatorname{tridiag}(-1,0,1),\quad({\bf C}_1\in\mathbb{R}^{N\times N}).
\]
With this definition, $-\overline{{\bf C}}_1 ^\top\overline{{\bf C}}_1$ and ${\bf C}_1$ are the discretized Laplace operator and central difference operator on $[-L,L]$,
with homogeneous Dirichlet boundary conditions. For $k=1,...,d$, define
\begin{align*}
	\overline{\bf C}_d^{(k)}&={\bf Id}^{(N)}_{k-1} \otimes \overline{\bf C}_1 \otimes {\bf Id}^{(N)}_{d-k},\notag\\
	{\bf C}_d^{(k)}&={\bf Id}^{(N)}_{k-1}  \otimes {\bf C}_1 \otimes {\bf Id}^{(N)}_{d-k}.
\end{align*}
Denote the grid points and mid-edge points in one direction by $x_i = (x_{i,1},...,x_{i,N})$ and $\overline{x}_k = (x_{i,\frac{1}{2}}, x_{i,\frac{3}{2}},...,x_{i,N+ \frac{1}{2}})$, respectively. The points tensor in the whole domain considered is obtained by 
\begin{align}\label{eq:grid-points}
	&{\mathcal{X}}={x}_1\times\cdots\times{x}_d,\\
	&\overline{\mathcal{X}}^{(k)}={x}_1\times\cdots\times\overline{x}_k\times\cdots\times{x}_d.
\end{align}
For functions $g:\mathbb{R}^d\rightarrow \mathbb{R}$, define the diagonal matrix ${\bf M}_g= \operatorname{diag}( g(\mathcal{X}))$. If $g=A_{kk}$, where $A$ is a matrix-valued function, then define $\overline{\bf M}_g = \operatorname{diag}(g(\overline{\mathcal{X}}^{(k)}))$.

The discretized versions of $\mathcal{L}_j$ ($0\le j\le d$) are given by
\begin{align}\label{eq:discretized_operator}
	{\bf L}_0&=\frac{1}{2}({\bf \Delta}_{\rho}+{\bf \Delta}_{G} )-{\bf C}_d-{\bf M}_{(0)},\notag\\
	{\bf L}_k &={\bf M}_{(k)}- \sum_{i}{\bf M}_{\rho_{ik}} {\bf C}_d^{(i)}.
\end{align}
where
\begin{align*}
	&{ \bf \Delta}_{*} = \sum_{i=1}^d -(\overline{\bf C}_d^{(i)})^\top \overline{\bf M}_{**^\top_{ii}}\overline{\bf C}_d^{(i)},\,{\bf C}_d=\sum_{k=1}^d{\bf M}_{f_k}{\bf C}_d^{(k)}, \\
	&{\bf M}_{(0)} = {\bf M}_{\operatorname{div}(f)}, \quad {\bf M}_{(k)} = {\bf M}_{h_k-\nabla\cdot(\rho_{\cdot k})}.\\
\end{align*}
\medskip
\noindent\textbf{TT approximation.} For efficiency, we replace ${\bf U}(t)$ and ${\bf L}_j$ with their TT approximations $\widetilde{\bf U}(t)$ and $\widetilde{\bf L}_j$, achieving relative accuracy $\varepsilon$. That is, the TT semi-discrete scheme is 
\begin{equation}\label{eq:TT semi-discrete scheme}
	d \widetilde{\bf U}(t)= \widetilde{\bf L}_0\widetilde{\bf U}(t)dt+\sum_{k=1}^d\widetilde{\bf L}_k\widetilde{\bf U}(t)dy_t^k,
\end{equation}
with 
\begin{align}\label{eq:relative accuracy}
	||\widetilde{{\bf U}}(0)-{\bf U}(0)||_F &\le \varepsilon ||{\bf U}(0)||_F, \\
	||\widetilde{{\bf L}}_0- {\bf L}_0||_F  &\le  \varepsilon (\Delta x)^{d/2+2}||{\bf L}_0||_F,\\
	||\widetilde{{\bf L}}_k - {\bf L}_k||_F &\le \varepsilon (\Delta x)^{d/2+1}||{\bf L}_k||_F.
\end{align}
The scaling with $(\Delta x)^{d/2}$ ensures that the absolute error is independent of the mesh size. Specifically, due to the sparsity of the matrix ${\bf L}_j$, $0\le j\le d$, we have
\begin{align}\label{eq:Frob_norm_bnd}
	||{\bf L}_0||_F &\le \frac{C_{{\bf L}}  }{(\Delta x)^{2+d/2}},\,||{\bf L}_k||_F \le  \frac{C_{{\bf L}}   }{(\Delta x)^{1+d/2}},
\end{align}
where $C_{{\bf L}}>0$ is independent of $\Delta x$. 

\medskip
\noindent\textbf{Time discretization.} We discretize the time interval $[0,T]$ by a uniform partition $0=t_0<t_1<t_2<\cdots <t_{N_T}=T$, where $t_{n+1}-t_n=\delta=\frac{T}{N_T}$. To achieve strong order one accuracy, we apply the semi-implicit Milstein method
\begin{align}\label{eq:implicit_Milstein}
	\hat{\bf U}_{n+1} =& \hat{\bf U}_n + \frac{\delta}{2}(\widetilde{\bf \Delta}_G+ \widetilde{\bf \Delta}_{\rho})\hat{\bf U}_{n+1}  -(\widetilde{\bf C}_d + \widetilde{\bf M}_{(0)})\hat{\bf U}_n\delta\notag\\
	&+ \sum_{j=1}^d \widetilde{\bf L}_j \hat{\bf U}_nI_{(j)}^n +\sum_{i,j=1}^d\widetilde{\bf L}^{(i,j)}\hat{\bf U}_nI_{(i,j)}^n,
\end{align}
where $I_{(j)}$, $I_{(i,j)}$ are defined in \eqref{eq:def_stochastic integral}. The implicit treatment of the diffusion operator requires solving a linear system. In our implementation, the inverse of the matrix ${\bf Id}_d^{(N)}-\frac{\delta}{2}(\widetilde{\bf \Delta}_G+ \widetilde{\bf \Delta}_{\rho})$ is pre-computed once. Although a sparse matrix does not imply sparsity of its inverse, numerical experiments indicate that the inverse has a low TT-rank (see Table \ref{tab:TT_rank_inverse}), which enables efficient implementation of the implicit method. The overall procedure is summarized in Algorithm \ref{alg:milstein_zakai}.

\begin{table}[ht]
	\centering
	\caption{TT-rank bound of $({\bf Id}_d^{(N)}- \frac{\delta} {2} ({\bf \Delta}_G  + {\bf \Delta_\rho} ))^{-1}$ 
		for different dimensions $d$ and DOF per spatial direction. 
		(In this example, $G$ is an orthogonal matrix, $\rho(x)=\operatorname{diag}(x)$. The stopping criterion for the Newton--Schulz iteration is 
		$\|\mathbf{A}_k^{-1}\mathbf{A} - \mathbf{I}\|_F / \|\mathbf{I}\|_F < 5\times 10^{-5}$.)}
	\label{tab:TT_rank_inverse}
	\begin{tabular}{ccccc}
		\toprule
		DOF per direction & $d=2$ & $d=4$ & $d=6$ & $d=8$ \\
		\midrule
		10  & 5 & 6 & 6 & 6\\
		20  & 7& 8 & 9 & 9 \\
		40  & 10 & 10 & 10 & 10 \\
		60  & 12 & 13 & 13 & 14 \\
		\bottomrule
	\end{tabular}
\end{table}

\begin{algorithm}[t]
	\caption{TT-based semi-implicit Milstein scheme for DMZ equation}
	\label{alg:milstein_zakai}
	\begin{algorithmic}[1]  % [1] adds line numbers
		
		\State \textbf{Offline stage: Precomputation}
		\State Compute the TT approximations $\widetilde{{\bf \Delta}}_G$, $\widetilde{{\bf \Delta}}_\rho$, $\widetilde{\bf C}_d$, $\widetilde{\bf M}_{(0)}$, and $\widetilde{\bf L}_k$, respectively, as defined in \eqref{eq:discretized_operator}, and set
		$\widetilde{\bf M}_R ={\bf Id} -\delta(\widetilde{\bf C}_d + \widetilde{\bf M}_{(0)})$ and $\widetilde{\bf L}^{(i,j)} = \widetilde{\bf L}_i \widetilde{\bf L}_j$.
		\State Compute $\widetilde{\bf M}_L:=({\bf Id}_d^{(N)}-\frac{\delta}{2}({\bf \Delta}_G + {\bf \Delta_\rho} ))^{-1}$ by the Newton-Schulz iteration in TT format. TT-rounding is used to control the growth of the TT-rank. 
		\State Approximate the initial distribution function $u_0$ by its discretized TT format $\widetilde{\bf U}_0$. 
		\Statex
		
		\State \textbf{Online stage: Time stepping}
		\State \textbf{Input:} $\widetilde{\bf M}_L$, $\widetilde{\bf M}_R$, $\widetilde{\bf L}_i$, $\widetilde{\bf L}^{(i,j)}$, and observations $\{y(t_n)\}$
		\State \textbf{Output:} Approximation $\hat{\bf U}_n$ of $u(t_n)$
		
		\For{$n = 0$ to $N_T-1$}
		
		\State Compute the stochastic integral approximations:
		\[
		\hat{I}_{(i,j)} = \frac{1}{2}\Delta y_{t_n}^i \Delta y_{t_n}^j,\quad ( i \neq j)
		,\qquad
		\hat{I}_{(i,i)} = \frac{1}{2}\left((\Delta y_{t_n}^i)^2 - \delta\right)
		\]
		
		\State Compute the update operator:
		\[
		{\bf M}_{\text{update}} =
		\sum_{j=1}^d \widetilde{\bf L}_j \Delta y_{t_n}^j
		+
		\sum_{i,j=1}^d \widetilde{\bf L}^{(i,j)} \hat{I}_{(i,j)}
		\]
		
		\State Update the solution:
		\[
		{\bf y} = \big(\widetilde{\bf M}_R + {\bf M}_{\text{update}}\big)\hat{\bf U}_n,
		\quad
		\hat{\bf U}_{n+1} = \widetilde{\bf M}_L {\bf y}
		\]
		
		\State Compute statistics with the approximation $\hat{\bf U}_n$
		\EndFor
	\end{algorithmic}
\end{algorithm}

\begin{remark}
	Suppose the TT-rank of the tensors in the online stage is bounded by $r$. 
	With mode size $N$ and dimension $d$, the computational cost is $\mathcal{O}(dNr^3)$ for tensor addition and $\mathcal{O}(dN^2r^4)$ for matrix--vector multiplication. 
	Since each online time step involves $\mathcal{O}(d^2)$ tensor additions and 2 matrix--vector multiplications, the overall complexity per step is $\mathcal{O}(d^3Nr^3 + dN^2r^4)$. 
	This polynomial dependence on $d$ indicates that the TT method alleviates the curse of dimensionality.
\end{remark}
\begin{remark}
	For $i \neq j$, the multiple stochastic integral
	\(
	I_{(i,j)}
	\)
	cannot be expressed solely in terms of the increments $I_{(i)}^n$ and $I_{(j)}^n$. 
	In practical implementations, a common approximation is given by
	\(	I_{(i,j)} \approx \frac{1}{2} I_{(i)}^n I_{(j)}^n\). The impact of this approximation on convergence is discussed in Section V.
\end{remark}

\section{Semi-discrete and TT Approximation Error Analysis}\label{sec:spatial convergence}
In this section, we analyze the errors arising from the spatial discretization and TT approximation. We first establish the convergence rate with respect to the mesh size, and then derive the error induced by TT truncation. A combined error estimate is given at the end of this section. 
\subsection{Spatial discretization error estimate}
We begin by analyzing structural properties of the discretized operator in \eqref{eq:discretized_operator} associated with the semi-discrete scheme \eqref{eq:semi-discrete scheme}. These properties are used to establish and derive the spatial convergence rate in Theorem \ref{thm:FD error} and stability estimate in Lemma \ref{lemma:stability estimate}.
\begin{lemma}\label{lemma:eig-bnd}
	For functions $g:\mathbb{R}^d\rightarrow \mathbb{R}$, if $g\in C^4$ with bounded derivatives, and $g\ge 0$, then there exist constants $h_0>0$ and $c_0>0$, independent of $\Delta x$, such that $\forall {\xi}\in\mathbb{R}^{N^d}$, $0<\Delta x\le h_0$ and $\overline{\bf M}_g^{(k)}=\operatorname{diag}(g( \overline{\mathcal{X}}^{(k)} ))$, 
	\[\xi^\top(-(\overline{\bf C}_d^{(k)})^\top \overline{\bf M}_{g}^{(k)}\overline{\bf C}_d^{(k)}	+({\bf C}_d^{(k)}) ^\top{\bf M}_{g}{\bf C}_d^{(k)})\xi<c_0(\xi^\top\xi).\]
\end{lemma}
\textbf{Proof} Here we give the proof for the $1$D case. The proof for $d>1$ can be directly generalized from this.
For $d=1$, define $d_j:=\xi_j-\xi_{j-1}$, and add the artificial boundary $\xi_0=\xi_{N+1}=0$. With this setting, 
\begin{align*}
	&\overline{\bf C}_1\xi = \frac{1}{\Delta x}(d_1, d_2 ,...,d_N, d_{N+1})^\top,\\
	&{\bf C}_1\xi =  \frac{1}{2\Delta x}(d_1+d_2, d_2+d_3,..., d_N+d_{N+1})^\top.
\end{align*}
Since $g\in C_b^4$, Taylor's expansion gives, uniformly in $i$,
\[  \frac{g(x_{i-1}) + g(x_i)}{2} = g(x_{i-1/2}) + \frac{(\Delta x)^2}{8} g''(x_{i-\frac{1}{2}}) + R_i, \]
where $|R_i|<\frac{(\Delta x)^4}{384}||g^{(4)}||_{L^\infty}$. Therefore, for $\Delta x<h_0$
\begin{align*}
	&\frac{1}{(\Delta x)^2}|-g(x_{i-1/2}) + \frac{g(x_{i-1}) + g(x_i)}{2}| \\
	=& | \frac{g''(x_{i-\frac{1}{2}}) }{8} + R_i|
	\le  \frac{||g''||_{L^\infty}}{8} + \frac{h_0^2}{384}||g^{(4)}||_{L^\infty}=:\frac{c_0}{2}
\end{align*}
By the positivity of $g$, the following estimate holds,
\begin{align*}
	&\xi^\top(-(\overline{\bf C})_1^\top\overline{\bf M}_{g}^{(k)}\overline{\bf C}_1)\xi + \xi^\top( {\bf C}_1)^\top{\bf M}_g{\bf C}_1\xi\\
	=& \frac{1}{(\Delta x)^2}\Big(-\sum_{i=1}^{N+1} g(x_{i-1/2})d_i^2 + \sum_{j=1}^{N}g(x_j)(\frac{d_j+d_{j+1}}{2})^2\Big)\\
	\le &  \frac{1}{(\Delta x)^2}\Big(-\sum_{i=1}^{N+1} g(x_{i-1/2})d_i^2 + \sum_{j=1}^{N} g(x_j) \frac{d_j^2+d_{j+1}^2}{2}\Big)\\
	\le &   \frac{1}{(\Delta x)^2}\sum_{i=1}^{N+1} (-g(x_{i-1/2}) + \frac{g(x_{i-1}) + g(x_i)}{2} )d_i^2 \\
	\le & 2\sum_{i=1}^{N+1} (\frac{g''(x_{i-\frac{1}{2}}) }{8}+ R_i)(\xi_i^2+\xi_{i-1}^2) \le c_0 \xi^\top\xi
\end{align*}

$\hfill\square$

Similar properties for the discretized Laplace operator and forward difference operators are proved in \cite{wise2009energy}. To establish stability for the semi-discrete system, we next estimate the logarithmic norm of certain matrix operators. For a square matrix $A$, the logarithmic norm is defined by $\mu(A):=\sup_{x\neq0} (x^\top Ax/x^\top x)$, or equivalently, $\mu(A)=\mu_{\max}(\frac{A+A^\top}{2})$.

\medskip
\begin{lemma}\label{lemma:lognorm-bnd}
	Let $g(x):{\bf I}\rightarrow{\mathbb{R}}$ be a Lipschitz continuous function, i.e., for $\forall x,y\in{\bf I}$, there exists a positive constant $L_g$ such that $|g(x)-g(y)|\le L_g|x-y|$.	Let ${\bf G}=\operatorname{diag} (g(\mathcal{X}))\in\mathbb{R}^{N^d\times N^d}$ be the diagonal matrix, where $\mathcal{X}$ is defined in \eqref{eq:grid-points}. Then we have the estimate
	\begin{align}
		\mu(\pm {\bf G}{\bf C}_d^{(k)}),\mu(\pm {\bf C}_d^{(k)}{\bf G})\le \frac{L_g}{2}.
	\end{align}
\end{lemma}

\textbf{Proof} It suffices to prove the bound for ${\bf C}_d^{(k)}{\bf G}$. By the definition of ${\bf C}_d^{(k)}$ in (\ref{eq:discretized_operator}), and its asymmetry, for ${\bf i},{\bf j}\in\{1,..,N\}^d$, we have
\begin{align*}
	\big({\bf C}_d^{(k)}{\bf G}+{\bf G}({\bf C}_d^{(k)})^\top\big)_{{\bf i},{\bf j}}=
	\begin{cases}
		\frac{g(\mathcal{X}_{{\bf i}+{\bf e}_k})-g(\mathcal{X}_{\bf i})}{2\Delta x},\quad &{\bf j}={\bf i}+{\bf e}_k,\\
		\frac{g(\mathcal{X}_{{\bf i}})-g(\mathcal{X}_{{\bf i}-{\bf e}_k})}{2\Delta x},\quad &{\bf j}={\bf i}-{\bf e}_k,\\
		0,&\text{others}.
	\end{cases}
\end{align*}
Therefore, the Gershgorin circle theorem yields
\begin{align*}
	\mu({\bf C}_d^{(k)}{\bf G})&=\lambda_{\max}(\frac{{\bf C}_d^{(k)}{\bf G}+{\bf G}({\bf C}_d^{(k)})^\top}{2})\\
	&\le \frac{1}{2}\cdot 2\cdot \frac{\max_{\bf i} (g(\mathcal{X}_{{\bf i}+{\bf e}_k})-g(\mathcal{X}_{\bf i}))}{2\Delta x}\le 
	%\frac{L_g\Delta x}{2\Delta x}=
	\frac{L_g}{2}.
\end{align*}
\hfill$\square$

\medskip
For the grid functions $u_{\bf i}$, ${\bf i}\in \{1,...,N\}^d$, we introduce the scaled discrete $L^2$ norm $||u||:=\sqrt{(\Delta x)^d \sum_{\bf i}u_{\bf i}^2}$, which is consistent with the continuous $L^2$ norm. The 2 norm of vector is defined as $|u|:=\sqrt{\sum_{\bf i}u_{\bf i}^2}$. Using the operator bounds established above, we now derive a stability estimate for the semi-discrete scheme \eqref{eq:semi-discrete scheme} in this norm.

With the stability result, Theorem \ref{thm:FD error} estimates the error between the semi-discrete solution ${\bf U}(t)$ and the exact solution $u(x,t)$ in $\Delta x$-scaling $L^2$ norm.
\medskip
\begin{theorem}\label{thm:FD error}
	Given smooth initial data $u(0,x)$, suppose that the strong solution $u(x,t)$ of the Zakai equation (\ref{eq:zakai}) and its derivatives with respect to $x$ up to order 4 are uniformly bounded on ${\bf I}$ for $0\le t\le T$, $T<\infty$. Define ${u}_{\bf i}(t)=u(t,x_{\bf i})$, $(u(\mathcal{X}))_{\bf i}=u_{\bf i}$, and the error $\mathcal{E}_{\bf i}(t)=u_{\bf i}(t)-{\bf U}_{\bf i}$. If $0<\Delta x<h_0$, then we have the error estimate for the semi-discrete scheme \eqref{eq:semi-discrete scheme},
	\begin{align*}
		\mathbb{E}(||\mathcal{E}(t)||^2_2)\le e^{K_0^{\text{FD}}t}(tK_1^{\text{FD}}(\Delta x)^4+\mathbb{E}(||\mathcal{E}(0)||^2_2)), 
	\end{align*}
	for some positive constants $K^{\text{FD}}_i$ independent of $\Delta x$.
\end{theorem}
\textbf{Proof.} With the regularity assumption that the spatial derivatives of $u$ are uniformly bounded up to order $4$, and by Taylor expansion with respect to physical variables $x$, the function $u(\mathcal{X},t)$ solves the SDE 
\begin{align}\label{eq:semi-discrete-TE}
	du = ({\bf L}_0u+\tau^0(t))\,dt+\sum_{k=1}^d({\bf L}_ku+\tau^k(t))\,dy_t^k,
\end{align}
where the truncation error $\tau_j,0\le j\le d$ satisfies
\begin{align}\label{eq:TE-bound}
	|\tau^j_{\bf i}(t)|\le A((\Delta x) ^2),
\end{align}
for all ${\bf i}, j$ and for some $A>0$ that depends on $f$ and $T$. Subtracting (\ref{eq:semi-discrete scheme}) from (\ref{eq:semi-discrete-TE}) yields the evolution of $\mathcal{E}(t)$
\begin{align}\label{eq:TE-evolv}
	d\mathcal{E}=({\bf L}_0\mathcal{E}+\tau^0)dt+\sum_{k=1}^d({\bf L}_k\mathcal{E}+\tau^k)dy_t^k.
\end{align}
By It\^o's formula and (\ref{eq:discretized_operator}), the evolution equation of $||\mathcal{E}(t)||^2$ is
\begin{align}
	&(\Delta x)^{-d}  d||\mathcal{E}(t)||^2\notag\\
	=& 2(\mathcal{E}^\top{\bf L}_0\mathcal{E} +\mathcal{E}^\top\tau^0)dt+2\sum_{k=1}^d(\mathcal{E}^\top{\bf L}_k\mathcal{E}+\mathcal{E}^\top\tau^k)dy_t^k+\sum_{k=1}^d|{\bf L}_k\mathcal{E}+\tau^k|^2dt \notag\\
	=& \big(\mathcal{E}^\top ({\bf \Delta}_\rho + {\bf \Delta}_G )  \mathcal{E} -2 \mathcal{E}^\top({\bf C}_d+{\bf M}_{(0)}) \mathcal{E}+2\mathcal{E}^\top\tau^0 \big)dt \notag\\
	&+2\sum_{k=1}^d(\mathcal{E}^\top {\bf M}_{(k)} \mathcal{E} - \sum_i\mathcal{E}^\top{\bf M}_{\rho_{ik}}{\bf C}_d^{(i)}\mathcal{E}  +\mathcal{E}^\top\tau^k)dy_t^k\notag\\
	&+\sum_{k=1}^d(\mathcal{E}^\top{\bf L}_k^\top{\bf L}_k\mathcal{E}+2\mathcal{E}^\top{\bf L}_k^\top\tau^k+|\tau^k|^2)dt \notag\\
	=& \mathcal{E}^\top(({\bf \Delta}_\rho + {\bf \Delta}_G ) +\sum_{k=1}^d{\bf L}_k^\top{\bf L}_k)\mathcal{E}\,dt
	-2\mathcal{E}^\top({\bf C}_d+{\bf M}_{(0)})\mathcal{E}dt\notag\\
	&+2 (\mathcal{E}^\top\tau^0+\sum_{k=1}^d\mathcal{E}^\top{\bf L}_k^\top\tau^k)dt+\sum_{k=1}^d|\tau^k|^2dt \notag\\
	&+ 2\sum_{k=1}^d (\mathcal{E}^\top({\bf M}_{(k)} - \sum_i {\bf M}_{\rho_{ik}} {\bf C}_d^{(i)}) \mathcal{E}+\mathcal{E}^\top\tau^k)dy_t^k .
\end{align}
Before estimating the global upper bound, we analyze some involved terms. First,
\begin{align}\label{eq:expand_LTL}
	\mathcal{E}^\top{\bf L}_k^\top{\bf L}_k\mathcal{E} = &\mathcal{E}^\top {\bf M}_{(k)}^2 \mathcal{E}- 2\sum_{i}\mathcal{E}^\top {\bf M}_{(k)}{\bf M}_{\rho_{ik}} {\bf C}_d^{(i)} \mathcal{E}  \notag\\
	&+ \mathcal{E}^\top \big( {\bf C}_d^{(i)} \big)^\top (\sum_{ij}{\bf M}_{\rho_{ik}} {\bf M}_{\rho_{jk}} ) {\bf C}_d^{( j)}\mathcal{E}. 
\end{align}
By the property of diagonal matrix and the definition of ${\bf M}_{(k)}$, we can bound by $\mathcal{E}^\top {\bf M}_{(k)}^2 \mathcal{E} < (C_h+dC_{\rho'} )^2 | \mathcal{E} |^2$, where $C_{\rho'}$ is the bound of first-order derivatives of $\rho$. Similarly, $ \big|\mathcal{E}^\top {\bf M}_{(0)} \mathcal{E}  \big|< dL_f |  \mathcal{E} |^2  $, $\big|\mathcal{E}^\top {\bf M}_{(k)}\mathcal{E} \big| <  (C_h + d C_{\rho'})  |  \mathcal{E}|^2$.	

A direct result of Lemma \ref{lemma:lognorm-bnd} and Assumption \ref{ass:matrix-regularity}, \ref{ass:regularity} yields that $ \big| \mathcal{E}^\top {\bf M}_{(k)}  {\bf M}_{\rho_{ik}} {\bf C}_d^{(i)}\mathcal{E} \big| < c_2|\mathcal{E}|^2$, with $c_2>0$ depends on $h$ and $\rho$. Similarly, $ \big|\mathcal{E}^\top {\bf C}_d \mathcal{E} \big|< \frac{dL_f}{2}|\mathcal{E}|^2$, $ \big|\mathcal{E}^\top {\bf M}_{\rho_{ik}} {\bf C}_d^{(i)} \mathcal{E} \big|< \frac{C_{\rho'}} {2} |\mathcal{E}|^2$.	

The summation over $k$ of last term on the right hand side in \eqref{eq:expand_LTL} can be further simplified with the definition of ${\bf M}_*$ and the Assumption \ref{ass:matrix-simplify},
\begin{align*}
	&\sum_k\mathcal{E}^\top \big( {\bf C}_d^{(i)} \big)^\top (\sum_{ij}{\bf M}_{\rho_{ik}} {\bf M}_{\rho_{jk}} ) {\bf C}_d^{(j)}\mathcal{E} \notag\\
	=&\sum_{ij} \mathcal{E}^\top \big( {\bf C}_d^{(i)} \big)^\top {\bf M}_{\rho\rho^\top_{ij}} {\bf C}_d^{(j)}\mathcal{E} 
	= \sum_i \mathcal{E}^\top \big( {\bf C}_d^{(i)} \big)^\top {\bf M}_{\rho\rho^\top_{ii}} {\bf C}_d^{(i)}\mathcal{E}. 
\end{align*}

By Cauchy's inequality and (\ref{eq:TE-bound}), for $0\le j\le d$, $1\le k\le d$,
\begin{align*}
	2\mathcal{E}^\top\tau^j&\le \mathcal{E}^\top\mathcal{E}+(\frac{2L}{\Delta x})^dA^2(\Delta x)^4,\\
	2\mathcal{E}^\top{\bf L}_k^\top\tau^k&\le c_1\, \mathcal{E}^\top{\bf L}_k^\top{\bf L}_k\mathcal{E}+\frac{1}{c_1}(\frac{2L}{\Delta x})^dA^2(\Delta x)^4,
\end{align*}
where $0<c_1 < \frac{\lambda_G}{\Lambda_\rho}$.

Integrating both sides over $[0,t]$, %using $\mathcal{E}(0)={\bf 0}$, 
taking the expectation, and applying Assumption \ref{ass:regularity}, \ref{ass:growth}, Lemma \ref{lemma:eig-bnd}, Lemma \ref{lemma:lognorm-bnd}, and the estimates mentioned above, we obtain
\begin{align}
	&(\Delta x) ^{-d}\Big( \mathbb{E}\big(||\mathcal{E}(t)||^2\big) -\mathbb{E}(||\mathcal{E}(0)||^2)\Big)\notag\\
	\le&
	\int_0^t \mathbb{E}\Big(  \mathcal{E}^\top( ({\bf \Delta}_\rho + {\bf \Delta}_G ) +(1+c_1)\sum_{k=1}^d{\bf L}_k^\top{\bf L}_k)\mathcal{E}\Big)\,dr  
	\notag\\
	&+\big( 3dL_f + 1 + C_hd(2C_h+3dC_{\rho'} +2 )\big)  \int_0^t  \mathbb{E}\big( |\mathcal{E}|^2\big)\,dr \notag\\
	&+\int_0^t (1+\frac{d}{c_1}+(1+C_h)d)(2L)^dA^2(\Delta x)^{-d+4} \,dr      \notag\\
	\le &\int_0^t\mathbb{E}\Big( \mathcal{E}^\top\big(  - \sum_{k=1}^d( \overline{\bf C}_d^{(k)})^\top \overline{\bf M}_{(GG^\top-c_1\rho\rho^\top)_{kk}} \overline{\bf C}_d^{(k)}   \big) \mathcal{E}   \Big) \,dr \notag\\
	&+(1+c_1)\int_0^t \sum_{k=1}^d\mathbb{E}\Big(\mathcal{E}^\top {\bf M}_{(k)}^2 \mathcal{E}   \Big)dr 
	-2(1+c_1)\int_0^t \sum_{i,k=1}^d\mathbb{E}\Big( \mathcal{E}^\top {\bf M}_{(k)}  {\bf M}_{\rho_{ik}} {\bf C}_d^{(k)}\mathcal{E} \Big)\,dr\notag\\
	&+\Big( 3dL_f + 1 + C_hd(2C_h+3dC_{\rho'} +2 ) 
	+ d(1+c_1)c_0\Big)  \int_0^t  \mathbb{E}\big( |\mathcal{E}|^2\big)\,dr \notag\\
	&+\int_0^t   (1+\frac{d}{c_1}+(1+C_h)d) (2L)^dA^2(\Delta x)^{-d+4} \,dr      \notag\\
	\le & \int_0^t K_0^{\text{FD}}\mathbb{E}\big( |\mathcal{E}|^2\big)\,dr 
	+tK_1^{\text{FD}}(\Delta x)^{-d+4},
\end{align}
where $K_0^{\text{FD}}=3dL_f + 1 + C_hd(2C_h+3dC_{\rho'} +2 )+ d(1+c_1)(c_0 + (C_h+dC_{\rho'})^2 + 2dc_2)$, $K_1^{\text{FD}}= 1+\frac{d}{c_1}+(1+C_h)d$ are independent of $\Delta x$. Apply the Gr{\" o}nwall's inequality, we have the error bound
\begin{align*}
	\mathbb{E}(||\mathcal{E}(t)||^2)\le e^{K_0^{\text{FD}}t}(tK_1^{\text{FD}}(\Delta x)^4+\mathbb{E}(||\mathcal{E}(0)||^2)).
\end{align*}
\hfill$\square$

\subsection{TT approximation error estimate}
In this subsection, we study the error introduced by the TT approximation. To control the propagation of the low-rank approximation error, we first establish a stability estimate of the semi-discrete scheme.

\begin{lemma}\label{lemma:stability estimate}
	%Given the smooth initial data $u(0,x)$, %and the assumptions \eqref{ass:matrix-simplify}, \eqref{ass:regularity}, \eqref{ass:growth}%Suppose as well that $q\ge(1+\alpha)\rho^2s$, $(\alpha>0)$, 
	The semi-discrete scheme (\ref{eq:semi-discrete scheme}) is stable, i.e., 
	\begin{align}\label{eq:stab_est}
		\mathbb{E}(||{\bf U}(t)||^2)\le e^{K_{\bf U}t}\mathbb{E}(||{\bf U}(0)||^2)
	\end{align}
	for some positive constant $K_{\bf U}$ independent of $\Delta x$.
\end{lemma}
\textbf{Proof.} Apply It\^o's formula, the norm $||{\bf U}(t)||^2$ satisfies
\begin{align*}
	&(\Delta x)^{-d} \,{d}||{\bf U}(t)||^2 \notag\\
	=&
	2{\bf U}^\top{\bf L}_0{\bf U}\,dt
	+2\sum_{k=1}^d{\bf U}^\top{\bf L}_k{\bf U}\,dy_t^k
	+\sum_{k=1}^d|{\bf L}_k{\bf U}|^2\,dt
	.
\end{align*}
Integrate both sides over $[0,t]$ and take the expectation, similarly to the proof in \ref{thm:FD error}, we obtain
\begin{align*}
	&(\Delta x)^{-d}\Big(
	\mathbb{E}(||{\bf U}(t)||^2)
	-\mathbb{E}(||{\bf U}(0)||^2)
	\Big) \notag\\
	\le&
	\int_0^t \mathbb{E}\Big(
	{\bf U}^\top\Big(({\bf \Delta}_\rho+{\bf \Delta}_G)
	+\sum_{k=1}^d{\bf L}_k^\top{\bf L}_k\Big){\bf U}
	\Big)\,dr 
	+2\int_0^t \mathbb{E}\Big(
	-{\bf U}^\top({\bf C}_d+{\bf M}_{(0)}){\bf U}
	\Big)\,dr \notag\\
	& +dC_h(2C_h+3dC_{\rho'})\int_0^t \mathbb{E}\big(  |{\bf U}|^2 \big) \,dr \notag\\
	\le	
	&  \int_0^t \mathbb{E}({\bf U}^\top {\bf \Delta}_G {\bf U})dr + \int_0^t \sum_{k=1}^d
	\mathbb{E}\Big(
	{\bf U}^\top{\bf M}_{(k)}^2{\bf U}
	\Big)\,dr 
	-2\int_0^t \sum_{i,k=1}^d
	\mathbb{E}\Big(
	{\bf U}^\top{\bf M}_{(k)}{\bf M}_{\rho_{ik}}{\bf C}_d^{(k)}{\bf U}
	\Big)\,dr \notag\\
	&+\big( dc_0+3dL_f+dC_h(2C_h+3dC_{\rho'})  \big)\int_0^t \mathbb{E}\big(  |{\bf U}|^2 \big) \,dr \notag\\
	\le&
	K_{\bf U}\int_0^t\mathbb{E}\big(|{\bf U}|^2\big)\,dr,
\end{align*}
where $K_{\bf U}=dc_0+3dL_f+dC_h(2C_h+3dC_{\rho'}) + d(C_h+dC_{\rho'})^2 + 2d^2 c_2 $ . Here we use Lemma \ref{lemma:eig-bnd}, \ref{lemma:lognorm-bnd}, and Assumptions \ref{ass:regularity}, \ref{ass:growth}. Finally, Gr{\"o}nwall's inequality yields the upper bound
\begin{align*}
	\mathbb{E}\big(||{\bf U}(t)||^2\big)
	\le e^{K_{\bf U}t}\mathbb{E}\big(||{\bf U}(0)||^2\big).
\end{align*}
\hfill $\square$

With the stability estimate of the semi-discrete scheme \eqref{eq:semi-discrete scheme}, we can now establish the TT approximation error.
\medskip
\begin{theorem}\label{thm:tt error}
	Suppose the assumptions in Lemma \ref{thm:FD error} are satisfied. With the   prescribed accuracy of TT approximation determined in (\ref{eq:relative accuracy}), we have the error estimate between the semi-discrete solution ${\bf U}(t)$ and the semi-discrete TT solution $\widetilde{\bf U}(t)$ of \eqref{eq:TT semi-discrete scheme}. Specifically, for the grid function $\widetilde{\mathcal{E}}_{\bf i}(t):={\bf U}_{\bf i}(t)-\widetilde{{\bf U}}_{\bf i}(t)$, the error estimate is
	\begin{align*}
		\mathbb{E}\big( ||\widetilde{\mathcal{E}}||^2 \big) \le e^{K^{\text{TT}}_0t}(\varepsilon^2 K^{\text{TT}}_1||{\bf U}(0)||^2+\mathbb{E}(||\widetilde{\mathcal{E}}(0)||^2)),
	\end{align*}
	where $K^{\text{TT}}_0, K^{\text{TT}}_1>0$ are independent of $\varepsilon$ and $t$.
\end{theorem}
\textbf{Proof.} Subtracting (\ref{eq:TT semi-discrete scheme}) from (\ref{eq:semi-discrete scheme}) and using It\^o's lemma, we have the evolution equation for $||\widetilde{{\mathcal{E}}}||^2$:
\begin{align*}
	&d||\widetilde{\mathcal{E}}||^2_2 \\
	=& 2(\Delta x)^d \Big[\Big( \widetilde{\mathcal{E}}^\top{\bf L}_0\widetilde{\mathcal{E}} - \widetilde{\mathcal{E}}^\top\big( {\bf L}_0-\widetilde{{\bf L}}_0 \big)\widetilde{\mathcal{E}}
	+ 
	\widetilde{\mathcal{E}}^\top\big( {\bf L}_0-\widetilde{{\bf L}}_0 \big){\bf U}\Big)dt+\sum_{k=1}^d\Big( \big|{\bf L}_k\widetilde{\mathcal{E}}\big|^2 + \big| \big( {\bf L}_k - \widetilde{{\bf L}}_k\big) \widetilde{\mathcal{E}} \big|^2 \notag\\
	&+ \big| \big( {\bf L}_k - \widetilde{{\bf L}}_k\big){\bf U} \big| ^2 - 2 \widetilde{\mathcal{E}}^\top {\bf L}_k^\top\big( {\bf L}_k - \widetilde{{\bf L}}_k \big) \big(\widetilde{\mathcal{E}} -{\bf U}\big) 
	- 2 \widetilde{\mathcal{E}}^\top \big( {\bf L}_k - \widetilde{{\bf L}}_k \big)^\top\big( {\bf L}_k - \widetilde{{\bf L}}_k \big){\bf U} \Big)dt\notag\\
	&+2\sum_{k=1}^d\Big( \widetilde{\mathcal{E}} ^\top{\bf L}_k\widetilde{\mathcal{E}} - \widetilde{\mathcal{E}} ^\top \big( {\bf L}_k-\widetilde{{\bf L}}_k \big)\widetilde{\mathcal{E}} + \widetilde{\mathcal{E}} ^\top \big( {\bf L}_k-\widetilde{{\bf L}}_k\big) {\bf U}\Big)dy_t^k \Big].
\end{align*}
Notice the fact that 
\(
{\bf x}^\top A{\bf x}\le ||A||_F|{\bf x}|^2
\)
, and
\(
|A{\bf x}|^2\le ||A||_F^2| {\bf x} |^2
\)
. Combine this with \eqref{eq:relative accuracy}, \eqref{eq:Frob_norm_bnd} we have
\(
\widetilde{\mathcal{E}}^\top ({\bf L}_j - \widetilde {\bf L}_j) \widetilde{\mathcal{E}} \le \varepsilon C_{\bf L} |\widetilde{\mathcal{E}} |^2
\),
\(
| ({\bf L}_j - \widetilde{{\bf L}}_j )\widetilde{\mathcal{E}} |^2 \le \varepsilon C_{\bf L} |\widetilde{\mathcal{E}} |^2
\). By Cauchy's inequality, we have 
\begin{align*}
	2\widetilde{{\mathcal{E}}}^\top ({\bf L}_j - \widetilde {\bf L}_j) {\bf U} \le |\widetilde{{\mathcal{E}}}|^2 + |  ({\bf L}_j - \widetilde {\bf L}_j) {\bf U}|^2,\\
	2\widetilde{{\mathcal{E}}}^\top {\bf L}_k^\top {\bf y} \le \frac{c_1}{2} |{\bf L}_k\widetilde{{\mathcal{E}}}|^2 + \frac{2}{c_1}|\bf y|^2.
\end{align*}
where $0< c_1< \frac{\lambda_G}{\Lambda_\rho}$. Similar to the upper bound estimate in the proof of Theorem \eqref{thm:FD error}, we have
\begin{align*}
	&(\Delta x)^{-d}\Big(\mathbb{E}\big( ||\widetilde{\mathcal{E}}(t)||^2 \big)-\mathbb{E} (||\widetilde{\mathcal{E}}(0)||^2)\Big)\notag\\
	\le &\int_0^t \mathbb{E}\Big(\widetilde{\mathcal{E}}^\top \big(2{\bf L}_0 +(1+c_1)\sum_{k=1}^d{\bf L}_k^\top{\bf L}_k \big)\widetilde{\mathcal{E}}  \Big)ds 
	+ \Big( 1+ C_h + dC_h(2C_h+3dC_\rho') \notag\\
	&+ \varepsilon C_{\bf L}(1+ ( 2+\frac{2}{c_1} ) d\varepsilon C_{\bf L} + 2dC_h) \Big) \int_0^t\mathbb{E}\big(|| \widetilde{\mathcal{E}} ||^2\big)ds
	+\varepsilon^2C^2_{{\bf L}}(1+2d+\frac{2d}{c_1}+C_h)\int_0^t \mathbb{E}\big(||{\bf U}||^2\big)ds \notag\\
	\le& \int_0^t\mathbb{E}\Big( \widetilde{\mathcal{E}}^\top\big(  - \sum_{k=1}^d( \overline{\bf C}_d^{(k)})^\top \overline{\bf M}_{(GG^\top-c_1\rho\rho^\top)_{kk}} \overline{\bf C}_d^{(k)}   \big) \widetilde{\mathcal{E}} \Big) \,dr    +2\int_0^t \mathbb{E}\Big(
	-{\widetilde{\mathcal{E}}}^\top({\bf C}_d+{\bf M}_{(0)})\widetilde{\mathcal{E}}
	\Big)\,dr      \notag\\
	&+\int_0^t \sum_{k=1}^d
	\mathbb{E}\Big(
	\widetilde{\mathcal{E}}^\top  \big( {\bf M}_{(k)}^2 - 2\sum_{i=1}^d{\bf M}_{(k)}{\bf M}_{\rho_{ik}}{\bf C}_d^{(k)}\big) \widetilde{\mathcal{E}}
	\Big)\,dr \notag\\
	&+ \Big( 1+ C_h + dC_h(2C_h+3dC_\rho') 	+ \varepsilon C_{\bf L}(1+ ( 2+\frac{2}{c_1} ) d\varepsilon C_{\bf L} + 2dC_h) \Big)\int_0^t\mathbb{E}\big(|| \widetilde{\mathcal{E}} ||^2\big)ds\notag\\
	&+\varepsilon^2C^2_{{\bf L}}(1+2d+\frac{2d}{c_1}+C_h)\int_0^t e^{K_{\bf U}s}||{\bf U}(0)||^2 ds \notag\\
	\le & K_0^{\text{TT}}\int_0^t\mathbb{E}\big(|| \widetilde {\mathcal {E}} ||^2\big)ds+\varepsilon^2K_1^{\text{TT}}||{\bf U}(0)||^2,
\end{align*}
where 
\begin{align*}
	K_0^{\text{TT}}=&1+ C_h + dC_h(2C_h+3dC_\rho') + 3dL_f  +d(C_h+dC_{\rho'}) + 2d^2c_2\\
	&+ \varepsilon C_{\bf L}(1+ ( 2+\frac{2}{c_1} ) d\varepsilon C_{\bf L} + 2dC_h) ,\\
	K_1^{\text{TT}}&=C^2_{{\bf L}}(1+2d+\frac{2d}{c_1}+C_h) \cdot\big(e^{K_{\bf U}t}-1\big)\frac{1}{K_{\bf U}}.
\end{align*}
By Gr{\" o}nwall's inequality, the error bound is obtained as
\begin{align*}
	\mathbb{E}\big(||\widetilde{\mathcal{E}}||^2\big)\le e^{K_0^{\text{TT}}t}(\varepsilon^2K_1^{\text{TT}}||{\bf U}(0)||^2+\mathbb{E} (|| \widetilde{\mathcal{E}}(0) ||^2) ).
\end{align*}
\hfill$\square$

\medskip
Combining Theorem \ref{thm:FD error} and Theorem \ref{thm:tt error}, the error estimate of the TT semi-discrete scheme \eqref{eq:TT semi-discrete scheme} is directly obtained.

\begin{theorem}
	Suppose the assumptions in Theorem \ref{thm:FD error} and Theorem \ref{thm:tt error} are all satisfied, and let $||{\bf U}(0)||=1$, then the TT semi-discrete solution $\widetilde{\bf U}(t)$ has error
	\begin{align*}
		&\mathbb{E}\big(||\widetilde{\bf U}(t)-u(t,\mathcal{X})||^2\big)\notag\\
		\le& tK_1e^{K_0t} \Big((\Delta x)^4 +\varepsilon^2 +\mathbb{E} ( || \mathcal {E} (0)|| ^2) +\mathbb{E} (|| \widetilde{\mathcal{E}}(0) ||^2) \Big),
	\end{align*}
	where $K_0, K_1>0$ are independent of $t$ and $\Delta x$.
\end{theorem}
\textbf{Proof.} By Theorem \ref{thm:FD error} and Theorem \ref{thm:tt error}, we have 
\begin{align*}
	&\mathbb{E}\big(||\widetilde{\bf U}(t)- u(t,\mathcal{X})||^2\big)
	\le 2\mathbb{E}\big( || \mathcal{E}(t) ||^2\big) + 2\mathbb{E}\big( ||\widetilde{ \mathcal{E}}(t) ||^2\big) \\
	\le & 2e^{K_0^{\text{FD}}t}(tK_1^{\text{FD}}(\Delta x)^4+\mathbb{E}(||\mathcal{E}(0)||^2)) 
	+ 2e^{K_0^{\text{TT}}t}(\varepsilon^2K_1^{\text{TT}}||{\bf U}(0)||^2+\mathbb{E} (|| \widetilde{\mathcal{E}}(0) ||^2) )\\
	\le & tK_1e^{K_0t} \Big((\Delta x)^4 +\varepsilon^2 +\mathbb{E} ( || \mathcal {E} (0)|| ^2) +\mathbb{E} (|| \widetilde{\mathcal{E}}(0) ||^2) \Big).
\end{align*}
\hfill$\square$

	\section{Time discretization error estimate}\label{sec:time convergence}

For a fixed mesh size $\Delta x$, we now try to derive the temporal convergence analysis. %Define the $\sigma$-algebra $\mathcal{F}_t$ generated by the observation path $\{y_s:0\le s\le t\}$. 
For simplicity, we assume that the prescribed TT accuracy $\varepsilon$ is small enough such that the TT approximation $\widetilde{\bf L}_j$ has the same bound for logarithmic norm and matrix norm as the original ones ${\bf L}_j$. Therefore, we write $\widetilde{\bf L}_j$ as ${\bf L}_j$ in this section. Throughout the section, $K$ is a generic constant independent of $\delta$.

The stability analysis of the numerical scheme relies on the estimate of stochastic multiple integrals of the observation process $y_t$.
\medskip
\begin{lemma}\label{lemma:integral_estimate}
	Let $\alpha=(j_1,\dots,j_l)$ be a multi-index with $0\le j_i\le d$, $0\le \tau < t \le \tau+\delta$, and assume $\delta < \frac{1}{C_h^2}$. 
	Then, for any $g\in\mathcal H_\alpha$,
	\begin{align}
		\mathbb{E}_\tau\!\left( |I_{\alpha}[g(\cdot)]_{\tau,t}|^2 \right)
		\le 4^{l(\alpha)-n(\alpha)} 
		\delta^{\,l(\alpha)+n(\alpha)}
		\mathcal{R}_{\tau,t}(g),
		\label{eq:sto-int-square-est}
	\end{align}
	where $$\mathcal{R}_{\tau,t}(g):=\sup_{\tau\le r\le t} 
	\mathbb{E}_\tau\left( |g(r)|^2 \right),$$ $l(\alpha)$, $n(\alpha)$ are defined in Section \ref{sec:preliminaries}. Consequently,
	\(
	|\mathbb{E}_\tau(I_{\alpha;\tau,t})|
	\le 2^{(l(\alpha)-n(\alpha))}
	\delta^{\frac12(l(\alpha)+n(\alpha))}.
	\)
	Moreover, for the semi-discrete solution ${\bf U}(\cdot)$,
	\(
	\mathbb{E}_\tau\!\left(
	|I_{\alpha}[{\bf U}(\cdot)]_{\tau,t}|^2
	\right)
	\le
	4^{l(\alpha)-n(\alpha)}
	\delta^{l(\alpha)+n(\alpha)}
	e^{K_{\bf U}\delta}
	|{\bf U}(\tau)|^2\).
\end{lemma}

\medskip
\begin{remark}
	The estimate order coincides with the classical bound for multiple It\^o integrals driven by Brownian motion.  
	The exponential growth in $l(\alpha)$ is harmless in the sequel, 
	as only multi-indices of bounded length appear in the It\^o-Taylor expansion.
\end{remark}
The proof of Lemma \ref{lemma:integral_estimate} follows the same inductive argument as Lemma 5.7.3 in \cite{2005Numerical}, with minor modifications to accommodate the drifted setting. To avoid distraction, we leave it in Appendix \ref{app:integral proof}. The following lemma provides a sharper estimate tailored to some specific stochastic multiple integrals.

\medskip
\begin{lemma}\label{lemma:first-integral-estimate}
	There exists a constant $K>0$, independent of $\delta$, such that for all $t\ge\tau$ with $t-\tau\le\delta$, and $1\le i,j,k\le d$,
	\begin{align*}
		&\big|\mathbb{E}_\tau[I{(k);\tau,t}]\big| \le K\delta, \\
		&\big|\mathbb{E}_\tau[I{(k,0);\tau,t}]\big| + \big|\mathbb{E}_\tau[I{(0,k);\tau,t}]\big| \le K\delta^2, \\
		&\big|\mathbb{E}_\tau[I{(i,j,k);\tau,t}]\big| \le K\delta^2.
	\end{align*}
\end{lemma}

\noindent\textbf{Proof. } (I) First consider $| \mathbb{E}_{\tau}(I_{(k);\tau,t})|$, by the zero conditional expectation of Brownian motion, we have
\begin{align*}
	&|\mathbb{E}_\tau(I_{(k);\tau,t})| %=& |\mathbb{E}(\int_{\tau_n}^{\tau_{n+1}}dy_s^k|\mathcal{F}_{\tau_n})|\notag\\
	=|\mathbb{E}_\tau(\int_{\tau}^{t}h_k(x_s)ds+ \int_{\tau} ^{t} dw_s )|\notag\\
	=&|\mathbb{E}_\tau(\int_{\tau}^{t}h_k(x_s)ds )| 
	\le \mathbb{E}_\tau(\int_{\tau}^{t}|h_k(x_s)| ds )\le C_h\delta.
\end{align*}

\noindent(II)  For $\mathbb{E}_\tau(I_{(k,0);\tau,t})$,
\begin{align*}
	&|\mathbb{E}_{\tau}(I_{(k,0)};\tau,t)|\\
	=&|\mathbb{E}_{\tau}(\int_{\tau}^{t}\int_{\tau}^{s_2} h_k(x_{s_1})ds_1ds_2 +\int_{\tau}^{t} \int_{\tau}^{s_2} dw_{s_1}ds_2)| \\
	\le & \frac{C_h}{2}\delta^2.
\end{align*}
where we use the zero first moment of multiple stochastic integrals of Brownian motion. The estimate of $\mathbb{E}_\tau(I_{(0,k);\tau,t})$ follows analogously, since the stochastic part again has zero first moment, and the deterministic part has the same bound.

\noindent(III) For $k_1,k_2,k_3=1,...,d$, using $dy_t^k = h_k(x_t)dt + dw_t^k$, the multiple stochastic integral $I_{(k_1,k_2,k_3);\tau,t}$ can be expanded as the sum of terms containing three $ds$, two $ds$ and one $dw$, one $ds$ and two $dw$, and three $dw$. According to Lemma 5.7.3 in \cite{2005Numerical}, the terms containing at least one $ds$ are bounded by $K\delta^2$.  The terms containing three $dw$ vanish since the stochastic multiple integral of Brownian motions has zero first moment. Therefore, we have
\begin{align*}
	&|\mathbb{E}_\tau(I_{(k_1,k_2,k_3);\tau,t}) ) |\notag\\
	%=& |\mathbb{E}_{\tau}(\int_{\tau}^{t} \int_{\tau}^{s_3} \int_{\tau}^{s_2} dy^{k_1}_{s_1} dy^{k_2}_{s_2}dy^{k_3}_{s_3})|\notag\\
	\le &K\delta^2 + |\mathbb{E}_{\tau}(\int_{\tau}^{t} \int_{\tau}^{s_3} \int_{\tau}^{s_2} dw^{k_1}_{s_1} dw^{k_2}_{s_2} dw^{k_3}_{s_3})| = K\delta^2
\end{align*}
\hfill$\square$

\medskip
The Milstein method achieves strong order one convergence for SDEs driven by Brownian motion; see \cite{2005Numerical} for detailed analysis. However, the semi-discrete DMZ equation \eqref{eq:semi-discrete scheme} is driven by a drifted observation process. As a result, the classical convergence analysis cannot be applied directly. To establish the convergence order of the semi-implicit Milstein method, we invoke the theorem proposed in \cite{mil1988theorem}. A key requirement of the theorem is the stability of the considered scheme. In particular, for a fixed mesh size $\Delta x$, the semi-implicit scheme is stable if the time step $\delta$ is small enough.

\medskip
\begin{lemma}\label{lemma:imMilstein_stability}
	Assume the stability conditions are satisfied:
	\begin{align}\label{eq:stab_cond}
		\delta<\frac{1}{C_h^2}
	\end{align}
	then the approximation $\hat{\bf U}_n$ obtained by the semi-implicit Milstein method \eqref{eq:implicit_Milstein} satisfies
	\begin{align}\label{eq:time_convergence}
		\mathbb{E}(|| \hat{\bf U}_{n}||^2)\le e^{K_{\text{stab}}\tau_n}\mathbb{E}\big(||\hat{\bf U}_0||^2\big),
	\end{align}
	where $K_{\text{stab}}$ is some positive constant independent of $\delta$.
\end{lemma}
\textbf{Proof.} Denote ${\bf Id}:={\bf Id}_d^{(N)}$ for simplicity. First consider the upper bound of $||({\bf Id}-\frac{\delta}{2}({\bf \Delta}_G + {\bf \Delta_\rho} ))^{-1}||_2$. Under Assumption~\eqref{ass:matrix-simplify}, ${\bf \Delta}_G + {\bf \Delta_\rho} $ is symmetric negative definite. Thus,
\(
\lambda_{\min}({\bf Id}-\frac{\delta}{2}({\bf \Delta}_G + {\bf \Delta_\rho} ))>1,
\)
and consequently
\begin{align}\label{eq:inv_norm bound}
	\left\|({\bf Id}-\frac{\delta}{2}({\bf \Delta}_G + {\bf \Delta_\rho} ))^{-1}\right\|_2\le 1.
\end{align}
In particular, it is bounded by $1$ uniformly with respect to $\Delta x$.
Now consider the bound for $\mathbb{E}\big(|({\bf Id}-{{\bf A}}_0\delta+\sum_k{\bf L}_kI_{(k)}^n +\sum_{k_1,k_2}{\bf L}_ {k_1} {\bf L}_{k_2}I_{(k_1,k_2)}^n) \hat{\bf U}_n|^2\big)$, where ${\bf A}_0={\bf C}_d+{\bf M}_{(0)}$. By the estimate in Lemma \ref{lemma:integral_estimate} and Lemma \ref{lemma:first-integral-estimate}, we have
\begin{align}\label{eq:RHS bound}
	&\mathbb{E}_{\tau_n}\big(|({\bf Id}-{\bf A}_0 \delta+\sum_k{\bf L}_kI_{(k)}^n +\sum_{k_1,k_2}{\bf L}^{(k_1,k_2)} I_{(k_1,k_2)}^n) \hat{\bf U}_n|^2\big) \notag\\
	& \le(1+K_{\text{stab}}\delta)|\hat{\bf U}_n|^2,
\end{align}
where $K_{\text{stab}}$ is a positive constant independent of $\delta$. The detailed derivation of \eqref{eq:RHS bound} refers to Appendix \ref{app:RHS estimate}, where the stability conditions \eqref{eq:stab_cond} are used to control the estimate. Using Lemma \ref{lemma:integral_estimate} and Lemma \ref{lemma:first-integral-estimate}, the conditional expectation of $I_{(k),n}$, $I_{(k_1,k_2),n}$ and their products are of order at most $\mathcal{O}(\delta)$ in conditional expectation. Hence, the whole expression admits the bound in \eqref{eq:RHS bound}. 

Combining \eqref{eq:implicit_Milstein}, \eqref{eq:inv_norm bound} and \eqref{eq:RHS bound}, we have
\begin{align}\label{eq:one_step}
	&\mathbb{E}\big( |\hat{\bf U}_{n+1}|^2 \big) =\mathbb{E}\Big( \mathbb{E}_{\tau_n}\big( |\hat{{\bf U}}_{n+1}|^2 \big) \Big) \notag\\
	\le & ||({\bf Id}-\frac{\delta}{2}({\bf \Delta}_G + {\bf \Delta_\rho} ))^{-1}||_2^2\mathbb{E}_{\tau_n}\big(|({\bf Id}-{\bf A}_0\delta)
	+\sum_k{\bf L}_kI_{(k)}^n +\sum_{k_1,k_2}{\bf L}^{(k_1,k_2)} I_{(k_1,k_2)}^n) \hat{\bf U}_n|^2\big)\notag\\
	\le & (1+K_{\text{stab}}\delta )\mathbb{E}\big(|{\hat{\bf U}_n}|^2\big).
\end{align}
Iterating the one-step estimate and using $(1+K_{\text{stab}}\delta)^n\le e^{K_{\text{stab}}\tau_n}$, we have
\begin{align*}
	\mathbb{E}\big( ||\hat{\bf U}_{n+1}||^2 \big)\le e^{K_{\text{stab}}\tau_{n+1} }\mathbb{E}\big(||\hat{\bf U}_0||^2\big).
\end{align*}
\hfill $\square$
\medskip

\begin{remark}
	From Appendix \ref{app:RHS estimate}, the constant $K_{\text{stab}}$ in the time-discretization convergence result \eqref{eq:time_convergence} depends on the spatial mesh size through the factor $\frac{1}{\Delta x}$. For a fixed mesh size and finite terminal time $T<\infty$, $K_{\text{stab}}$ remains finite. Such dependence is standard in the convergence analysis in the time direction of numerical schemes for the semi-discrete scheme arising from spatial discretization. In fact, the convergence of general SDE numerical schemes relies on the Lipschitz constants of the drift and diffusion terms. In the case we consider, the Lipschitz constants are essentially the matrix norms of the discretized spatial operators, which are typically scaled by $\frac{1}{\Delta x}$. A similar result can be found in \cite{sheng2023nonconventional}. Therefore, the present result focuses on time-discretization convergence under fixed spatial resolution, while spatial convergence is addressed separately.
\end{remark}

\medskip

The convergence of the semi-implicit Milstein scheme \eqref{eq:implicit_Milstein} is established using the theorem provided in \cite{mil1988theorem}.

\medskip

\begin{lemma}\label{lemma:truncation-whole error}
	Let ${\bf U}_{t_k,{\bf X}}(t_i)$ be the exact solution of 
	\eqref{eq:semi-discrete scheme} with initial condition ${\bf U}(t_k)={\bf X}$, and 
	$\hat{\bf U}_{t_k,{\bf X}}(t_i)$ its stable numerical approximation at time $t_i$ under the condition that $\hat{\bf U}_i = {\bf X}$. For simplicity, we write ${\bf U}(t):={\bf U}_{t_0,{\bf U}(0)}(t)$, $\hat{\bf U}_n:=\hat{\bf U}_{t_0,{\bf U}(0)}(\tau_n)$. Assume that the numerical scheme is stable:
	\begin{align}
		\mathbb{E}(|\hat{\bf U}_n|^2)
		\le K\big(1+\mathbb{E}(|\hat{\bf U}(0)|^2)\big).
	\end{align}
	Assume also that the local one-step errors satisfy, for all $n=0,1,\dots,N-1$,
	\begin{align*}
		\left|		\mathbb{E}_{\tau_n}\big(		{\bf U}_{\tau_n, {\bf X}}(\tau_{n+1})-\hat{\bf U}_{\tau_n, {\bf X}}(\tau_{n+1})
		\big)		\right|
		&\le		K \sqrt{1+|{\bf X}|^2}
		\delta^{p_1},	%	\label{eq:expected deviation clean}
		\\		\mathbb{E}_{\tau_n}\big(
		|{\bf U}_{\tau_n, {\bf X}}(\tau_{n+1})-\hat{\bf U}_{\tau_n, {\bf X}}(\tau_{n+1})|^2
		\big)			&\le		K \big(1+|{\bf X}|^2\big)
		\delta^{2p_2},		%\label{eq:mean-square error clean}
	\end{align*}
	where $p_1 \ge p_2+\tfrac{1}{2}$ and $p_2 \ge \tfrac{1}{2}$. Then, for all $n=0,1,\dots,N$,
	\begin{align}
		\left[		\mathbb{E}\big(		|{\bf U}(\tau_n)-\hat{\bf U}_n|^2
		\big)		\right]^{1/2}		\le		K\big(1+\mathbb{E}(|{\bf U}(0)|^2)\big)^{1/2}
		\delta^{\,p_2-\frac12}.
	\end{align}
\end{lemma}

\medskip
As a combination of Lemma \ref{lemma:imMilstein_stability} and Lemma \ref{lemma:truncation-whole error}, we have the convergence result for the semi-implicit Milstein scheme \eqref{eq:implicit_Milstein}.
\medskip
\begin{theorem}\label{thm:converge order of imMilstein}
	If the stability conditions in Lemma \ref{lemma:imMilstein_stability} are satisfied, then the ideal
	semi-implicit Milstein scheme with exact iterated stochastic integrals has strong
	convergence order one, i.e.,
	\[
	|\mathbb{E}(|{\bf U}(\tau_n)-\hat{\bf U}_n|^2)|^{1/2} \le K(1+\mathbb{E}(|{\bf U}(0)|^2))^{1/2}\delta.
	\]
\end{theorem}
\medskip
\textbf{Proof.} By Lemma \ref{lemma:truncation-whole error}, it suffices to derive the one-step expected deviation and mean-square error of the scheme. The It\^o-Taylor expansion of ${\bf U}(\tau_{n+1})$ is given by
\begin{align}\label{eq:Ito-Taylor expn}
	{\bf U}(\tau_{n+1}) 
	= &{\bf U}(\tau_n) + \delta {\bf L}_0 {\bf U}(\tau_n) + \sum_{k=1}^d {\bf L}_kI_{(k)}^n {\bf U}(\tau_n) 
	+ \sum_{k_1,k_2}{\bf L}^{(k_1,k_2)}I_{(k_1,k_2)}^n{\bf U}(\tau_n) + \frac{\delta^2}{2}{\bf L}_0^2{\bf U}(\tau_n) \notag\\
	&+\sum_{k}\big( {\bf L}^{(0,k)}I_{(0,k)}^n + {\bf L}^{(k,0)} I_{(k,0)}^n \big) {\bf U}(\tau_n)+ \sum_{k_1,k_2,k_3}{\bf L}^{(k_1,k_2,k_3)}I_{(k_1,k_2,k_3)}^n{\bf U}(\tau_n)+R,
\end{align}
where $R=\sum_{\alpha\in \mathcal{B}(\mathcal{A}_{1.5})}{\bf L}^{\alpha}I_{\alpha}[{\bf U}(\cdot)]_{\tau_n,\tau_{n+1}}$, then by Lemma \ref{lemma:integral_estimate}, $R$ is bounded by $|\mathbb{E}_{\tau_n}(R)|\le K(1+|{\bf U}(\tau_n)|^2)^{1/2}\delta^{2}$, $\mathbb{E}_{\tau_n}(|R|^2)\le K(1+|{\bf U}(\tau_n)|^2)\delta^4$. For simplicity, denote $ {\bf \Delta}_d:=\frac{1}{2}\big({\bf {\Delta}}_G + {\bf \Delta}_\rho\big)$.
Apply the It\^o-Taylor expansion to ${\bf \Delta}_d{\bf U}(\tau_{n+1})$, we have
\begin{align}\label{eq:L*U-Ito-Taylor expn}
	&{\bf \Delta}_d{\bf U}(\tau_{n+1}) 
	=  {\bf \Delta}_d \Big( {\bf U}(\tau_n) + \delta{\bf L}_0 {\bf U}(\tau_n) + \sum_{k=1}^d  {\bf L}_kI_{(k)}^n  {\bf U}(\tau_n)\Big)+R_1,
\end{align}
where $R_1={\bf \Delta}_d\sum_{\alpha\in\mathcal{B}(\mathcal{A}_{0.5})}I_{\alpha}[{\bf U}(\cdot)]_{\tau_n,\tau_{n+1}}$, then by Lemma \ref{lemma:integral_estimate}, we have $|\mathbb{E}_{\tau_n}(R_1)|\le K(1+|{\bf U}(\tau_n)|^2)^{1/2}\delta$, $\mathbb{E}_{\tau_n}(|R_1|^2)\le K(1+|{\bf U}(\tau_n)|^2)\delta^2$.
Substitute \eqref{eq:L*U-Ito-Taylor expn} into \eqref{eq:Ito-Taylor expn}, the ${\bf U}(\tau_{n+1})$ is expressed in
\begin{align}\label{eq:Im-Ito-Taylor expansion}
	&{\bf U}(\tau_{n+1})={\bf U}(\tau_n) + \delta{\bf \Delta}_d{\bf U}(\tau_{n+1})-\delta({\bf C}_d+{\bf M}_{(0)}) {\bf U}(\tau_n) \notag\\
	&+ \big( \sum_{k=1}^d {\bf L}_kI_{(k)}^n+ \sum_{k_1,k_2}{\bf L}^{(k_1,k_2)}I_{(k_1,k_2)}^n\big){\bf U}(\tau_n)+R_2
\end{align}
where 
\begin{align*}
	R_2 = &  \frac{\delta^2}{2}{\bf L}_0^2{\bf U}(\tau_n) +\sum_{k}\big( {\bf L}^{(0,k)}I_{(0,k)}^n + {\bf L}^{(k,0)} I_{(k,0)}^n \big) {\bf U}(\tau_n)\notag\\
	&+ \sum_{k_1,k_2,k_3}{\bf L}^{(k_1,k_2,k_3)}I_{(k_1,k_2,k_3)}^n{\bf U}(\tau_n)+R,\\
	&
	+ \delta{\bf \Delta}_d \Big(  \delta{\bf L}_0 {\bf U}(\tau_n) + \sum_{k=1}^d  {\bf L}_kI_{(k)}^n  {\bf U}(\tau_n)\Big)+\delta R_1.
\end{align*}
By the integral estimate in Lemma, \eqref{lemma:integral_estimate} (\ref{lemma:first-integral-estimate}), we have $|\mathbb{E}_{\tau_n}(R_2)|\le K(1+|{\bf U}(\tau_n)|^2)^{1/2}\delta^2$, $\mathbb{E}_{\tau_n}(|R_2|^2)\le K(1+|{\bf U}(\tau_n)|^2)\delta^3$. 

For local truncation error, assume that $\hat{\bf U}_{n} = {\bf U}(\tau_n)$. Subtract \eqref{eq:implicit_Milstein} from \eqref{eq:Im-Ito-Taylor expansion}, 
\begin{align*}
	{\bf U}(\tau_{n+1})-\hat{\bf U}_{n+1} = \delta{\bf \Delta}_d({\bf U}(\tau_{n+1})-\hat{\bf U}_{n+1}) + R_2 ,
\end{align*}
by the operator norm estimate in Lemma \ref{lemma:imMilstein_stability}, we have
\begin{align*}
	&|\mathbb{E}_{\tau_n}({\bf U}_{\tau_n,{\bf U}(\tau_n)}(\tau_{n+1})-\hat{\bf U}_{\tau_n,{\bf U}(\tau_n)}(\tau_{n+1}))|\notag\\
	\le& |\mathbb{E}_{\tau_n} (R_2)| \le  K\sqrt{1+|{\bf U}(\tau_n)|^2}\delta^2.
\end{align*}
Similarly, 
\begin{align*}
	&\mathbb{E}_{\tau_n}(|{\bf U}_{\tau_n,{\bf U}(\tau_n)} (\tau_{n+1})-\hat{\bf U}_{\tau_n,{\bf U}(\tau_n)}(\tau_{n+1})|^2)\\
	\le& \mathbb{E}_{\tau_n}|R_2|^2\le  K(1+|{\bf U}(\tau_n)|^2)\delta^3.
\end{align*}
By Lemma \ref{lemma:truncation-whole error} and the stability estimate in Lemma \ref{lemma:imMilstein_stability}, the convergence order of the semi-implicit method is 1, i.e., $|\mathbb{E}(|{\bf U}(\tau_n)-\hat{\bf U}_n|^2)|^{1/2} \le K(1+\mathbb{E}(|{\bf U}(0)|^2))^{1/2}\delta$.
\hfill$\square$
~\\
\begin{remark}
	The convergence analysis in Section \ref{sec:time convergence} assumes that the stochastic integrals $I_{(i,j)}$ are available exactly. This separates the temporal
	discretization error associated with the It\^o-Taylor expansion from the additional error	introduced by approximating these integrals. In practical implementations, the terms $I_{(i,j)}$, $i\neq j$, are usually approximated, for example by	\(
	I_{(i,j)}^n \approx \frac{1}{2} I_{(i)}^n I_{(j)}^n ,
	\)
	which may reduce the strong temporal convergence order to $1/2$. Hence, Theorem \ref{thm:converge order of imMilstein} should be interpreted as a convergence result
	for the ideal Milstein-type scheme with exact iterated stochastic integrals. Its main purpose is to prove the convergence rate of It\^o-Taylor-based discretization for a drifted observation process without a change of probability measure. In the implementation, we still adopt the Milstein-type scheme because it is observed to provide better robustness than the
	Euler method.
\end{remark}

\section{Numerical Experiments}
\label{sec:numerical result}

\subsection{Cubic sensor example}\label{subsec: cubic sensor}
The cubic sensor problem is defined by the signal model with a cubic observation function:
\begin{align}\label{eq:cubic_sensor_prob}
	\begin{cases}
		dx_t = (A_dx_t+\sin(1.5x_t))dt+\frac{1}{\sqrt{5}}dv_t+\frac{\sqrt{2}} {5}dw_t,\\
		dy_t^{(k )}= (x_t^{(k)})^3dt+dw_t^{(k)}, \quad k=1,...,d.
	\end{cases}
\end{align}
where $A_d\in \mathbb{R}^{d\times d}$, and
\begin{align}
	(A_d)_{ij}=\begin{cases}
		-0.8, \quad i=j,\notag\\
		-0.1, \quad i-j=1,\notag\\
		0, \quad \text{others},
	\end{cases}
\end{align}
The highly nonlinear observation model makes it difficult to estimate the state of the system. Here we truncate Zakai's equation \eqref{eq:zakai} within the cube $[-2.2,\,2.2]^d$, and the degree of freedom (DOF) in each direction is $N=8$. The terminal time is $T=20\text{s}$, with time step $\delta=0.01\text{s}$. The estimation error is measured by the root mean squared error (RMSE):
\[ \text{RMSE}:=\frac{1}{\sqrt{dN_T}}\sqrt{\sum_{k=1}^{N_{T}}  | x(t_k)-\hat{x}(t_k) |^2 }\]

For each dimension, 100 independent trials were conducted to evaluate the performance of the proposed TT method and the PF. The average RMSE and online CPU time are summarized in Table \ref{tab:cubic_TTvsPF}. Notably, due to the strong nonlinearity of the observation function $h$ and the violation of the Gaussian assumption, the EKF exhibited extreme numerical instability, with a failure rate exceeding $97\%$ across all trials; hence, its results are omitted from the comparison. 

To ensure a fair evaluation, the PF method was implemented with varying particle counts. While the computational cost of PF scales with the number of particles, its estimation accuracy deteriorates as the dimension increases. As evidenced in Table \ref{tab:cubic_TTvsPF}, even with a significantly large number of particles, PF consistently yields higher RMSE than the TT method, while incurring substantially higher computational burden. For instance, in the 7D case, the TT method achieves an RMSE of 0.4472 in 21.19s, whereas PF with 11,000 particles still has a higher RMSE (0.4920) and requires 128.88s. Moreover,  for a fixed number of particles, the performance of PF degrades noticeably as the dimension increases. For example, with 7,000 particles, the RMSE of PF increases from 0.5064 in 7D to 0.5923 in 9D. A similar trend is observed with 9,000 particles, where the RMSE increases from 0.4960 to 0.5542. 

In contrast, the proposed TT method maintains stable accuracy across different dimensions, demonstrating its robustness in high-dimensional settings. Compared with PF, it achieves a more favorable balance between accuracy and computational cost. Notably, even when the number of particles is increased, PF exhibits significant degradation in performance as the dimension grows, whereas the TT method remains stable. These results demonstrate that the TT-based approach provides an effective and real-time solution for medium- to high-dimensional nonlinear filtering problems where traditional methods fail or become computationally prohibitive.

The prediction results of the TT method, PF, and EKF in a single trial for the 7D problem are depicted in Fig. \ref{fig:cubic_prob}. In this trial, the EKF diverges numerically from around $t=3s$.
\subsection{Multi-mode example}
For the multi-mode case, we consider the problem
\begin{align}\label{eq:multi_mode_prob}
	\begin{cases}
		dx_t^{(1) }=\big(0.5\sin x_t^{(1)}+ 0.2\sin x_t^{(2)}\big)dt +d\widetilde{v}_t^{(1)},\\
		dx_t^{(2)} = \big(0.6\sin x_t^{(2 )}+0.2\sin x_t^{(3)} \big)dt+d\widetilde{v}_t^{(2)},\\
		dx_t^{(3) }= \big(0.8\sin x_t^{(3)}+0.2\sin x_t^{(4)}  \big)dt+d\widetilde{v}_t^{(3)},\\
		dx_t^{(4)}= 0.5\sin x_t^{(4)} dt+d\widetilde{v}_t^{(4)} \\
		dy_t^{(k) }= \big(0.2(x_t^{(k)})^2+\cos(0.6x_t^{(k)})\big)dt+dw_t^{(k)},
	\end{cases}
\end{align}
where $k=1,...,4$, $d\widetilde{v}_t^{(k)} = \frac{\sqrt{30}} {10}dv_t^{(k )}+ \frac{1}{5\sqrt{10}}x_t^{(k)}dw_t^{(k)}$. This experiment solves the problem within the cube $[-4.5,4.5]^4$ and over time interval $[0,20]$. Set the DOF at each direction as $N=25$ and time step as $\delta=0.01s$.

The drift function $f$ and correlation function $\rho$ are odd functions, and the observation function $h$ is an even function. Given these settings, it can be readily verified that the posterior marginal density functions of $x_t$ are bimodal and symmetric with respect to the $x$-axis. 

In the numerical experiment, the TT method recovers the symmetric bimodal density function since it directly computes the density function, which can be seen from the heatmap in Fig. \ref{fig:tt_4d}. However, each particle only randomly chooses one possible mode to predict the state. Although we have multiple particles in the PF to approximate the distribution, numerical results show that it lacks the ability to recover the symmetric distribution due to randomness. Specifically, the heatmap in Figure \ref{fig:pf_4d} shows that the approximated density tends to concentrate near a certain mode, deviating from the exact posterior density function. As for the EKF, it exhibits instability in the multi-mode example and diverges from the ground truth state. 

\begin{table}[htbp]
	\centering
	\caption{Average RMSE and Online CPU Time of the TT method and PF on problem \eqref{eq:cubic_sensor_prob} ($d \in \{5, 7, 9\}$) over 100 trials.}
	\label{tab:cubic_TTvsPF}
	\renewcommand{\arraystretch}{1.2} 
	\begin{tabular}{llcc}
		\toprule
		\textbf{Dimension} ($d$) & \textbf{Method} & \textbf{RMSE} & \textbf{Avg. Time (s)} \\
		\midrule
		\multirow{4}{*}{5D} & \textbf{TT (Proposed)} & \textbf{0.4301} & \textbf{9.69} \\
		& PF (2k particles)     & 0.5123          & 8.23 \\
		& PF (3k particles)     & 0.4772          & 14.74 \\
		& PF (4k particles)     & 0.4762          & 22.84 \\
		\midrule
		\multirow{4}{*}{7D} & \textbf{TT (Proposed)} & \textbf{0.4472} & \textbf{21.19} \\
		
		& PF (5k particles)     & 0.5223          & 34.64 \\
		& PF (7k particles)     &  0.5064         & 59.51 \\
		& PF (9k particles)     &  0.4960        & 91.13 \\
		& PF (11k particles)     &  0.4920        & 128.88 \\
		
		\midrule
		\multirow{4}{*}{9D (T=10s)} & \textbf{TT (Proposed)} & \textbf{0.4691} & \textbf{18.51} \\
		& PF (5k particles)     &   0.5999        & 18.13  \\
		& PF (7k particles)     &   0.5923        &  30.98\\
		& PF (9k particles)     &   0.5542        & 47.04\\
		\bottomrule
	\end{tabular}
\end{table}

\begin{figure}[htbp]
	\centering
	\begin{subfigure}{0.24\textwidth}
		\centering
		\includegraphics[width=\linewidth]{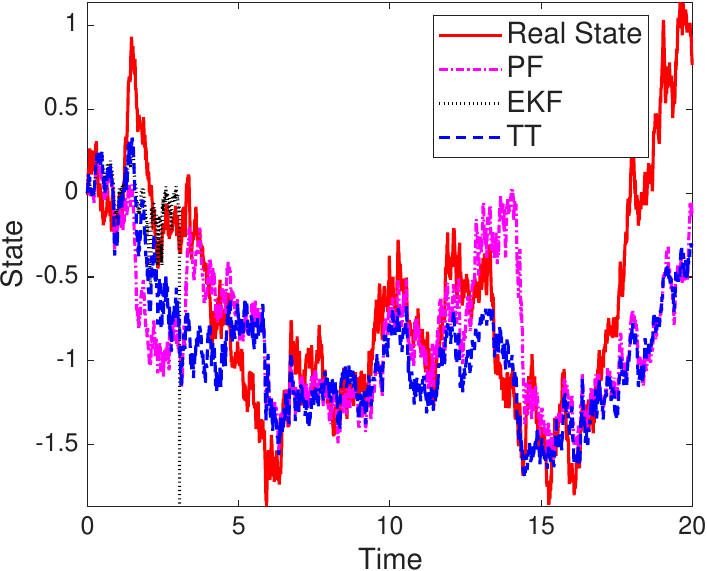}
		\caption{1st Dimension}
		\label{fig:dim1}
	\end{subfigure}
	\hfill
	\begin{subfigure}{0.24\textwidth}
		\centering
		\includegraphics[width=\linewidth]{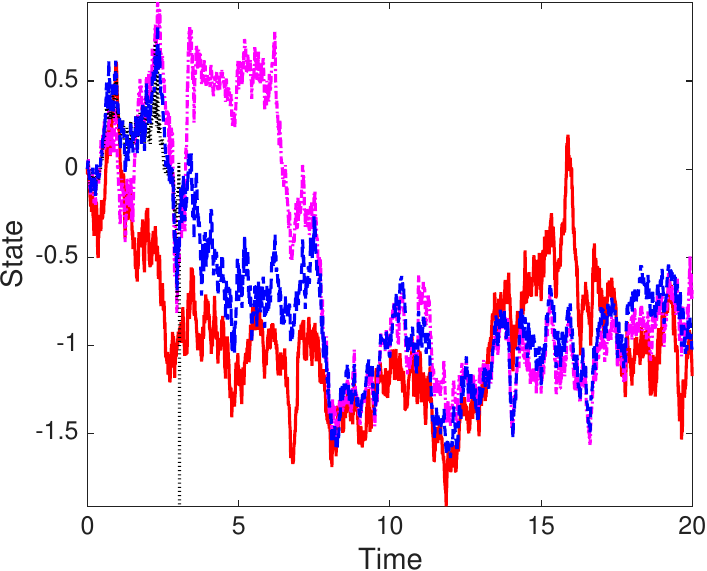}
		\caption{3rd Dimension}
		\label{fig:dim3}
	\end{subfigure}
	\hfill
	\begin{subfigure}{0.24\textwidth}
		\centering
		\includegraphics[width=\linewidth]{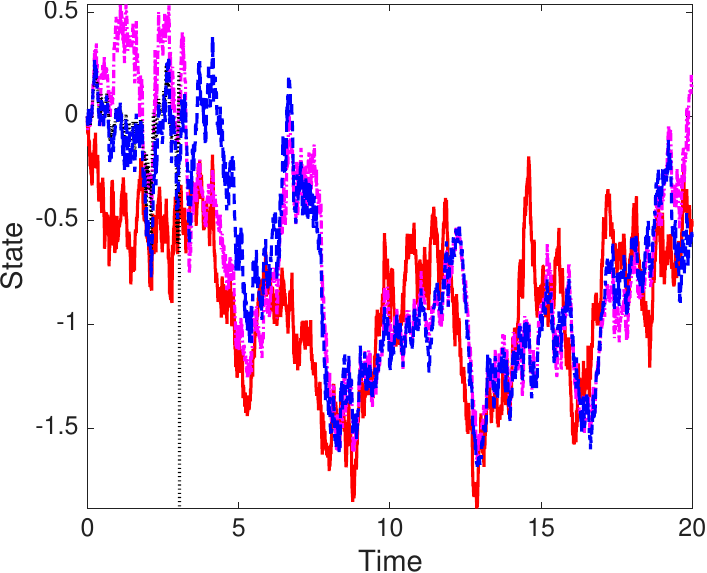}
		\caption{5th Dimension}
		\label{fig:dim5}
	\end{subfigure}
	\hfill
	\begin{subfigure}{0.24\textwidth}
		\centering
		\includegraphics[width=\linewidth]{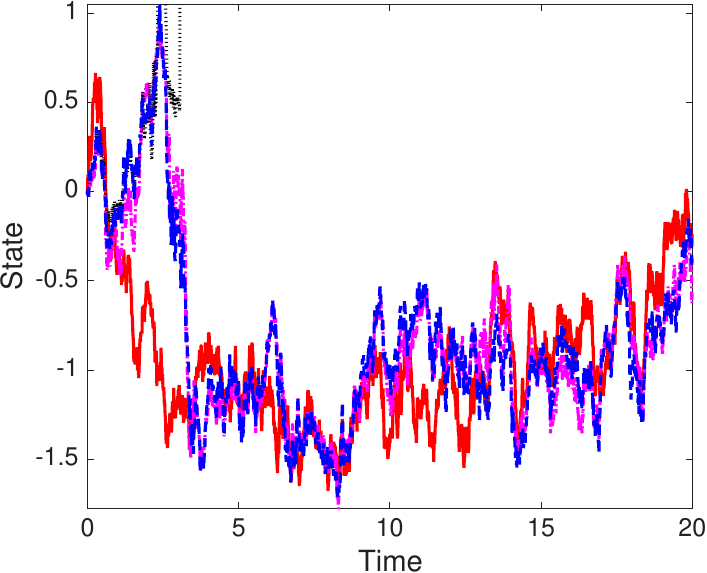}
		\caption{7th Dimension}
		\label{fig:dim7}
	\end{subfigure}
	
	\caption{Tracking results of the TT method, PF (5k particles), and EKF in a single trial for the 7D cubic sensor problem \eqref{eq:cubic_sensor_prob}.}
	\label{fig:cubic_prob}
\end{figure}

\begin{figure}[htbp]
	\centering
	\begin{subfigure}{0.24\textwidth}
		\centering
		\includegraphics[width=\linewidth]{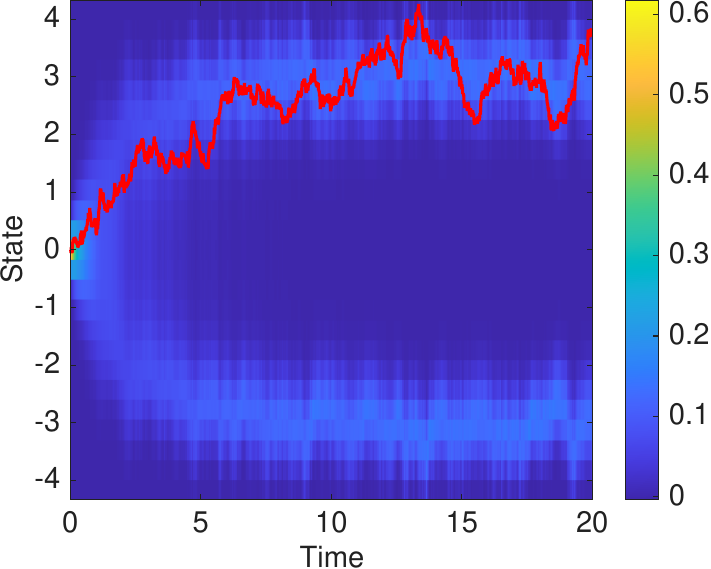}
		\caption{1st Dimension}
		\label{fig:tt_dim1}
	\end{subfigure}
	\hfill
	\begin{subfigure}{0.24\textwidth}
		\centering
		\includegraphics[width=\linewidth]{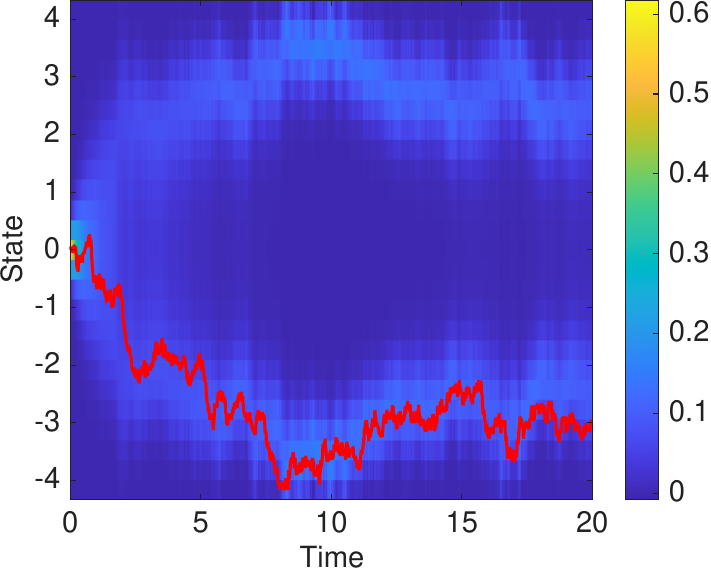}
		\caption{2nd Dimension}
		\label{fig:tt_dim2}
	\end{subfigure}
	\hfill
	\begin{subfigure}{0.24\textwidth}
		\centering
		\includegraphics[width=\linewidth]{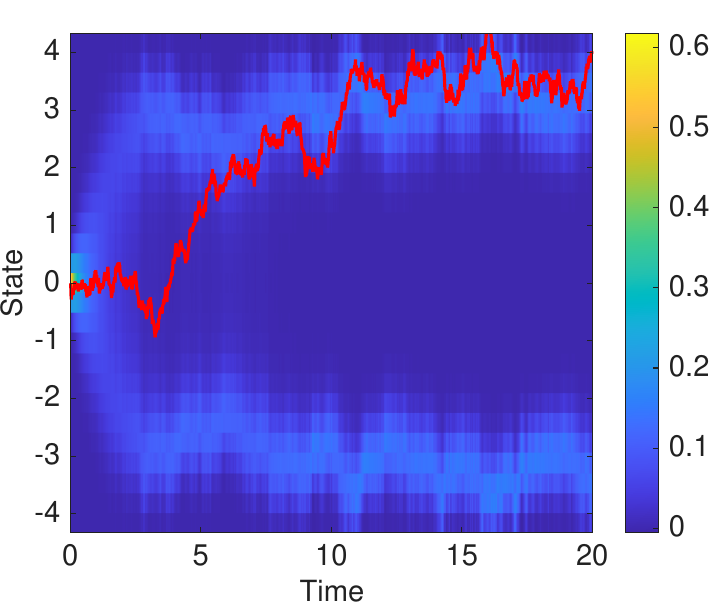}
		\caption{3rd Dimension}
		\label{fig:tt_dim3}
	\end{subfigure}
	\hfill
	\begin{subfigure}{0.24\textwidth}
		\centering
		\includegraphics[width=\linewidth]{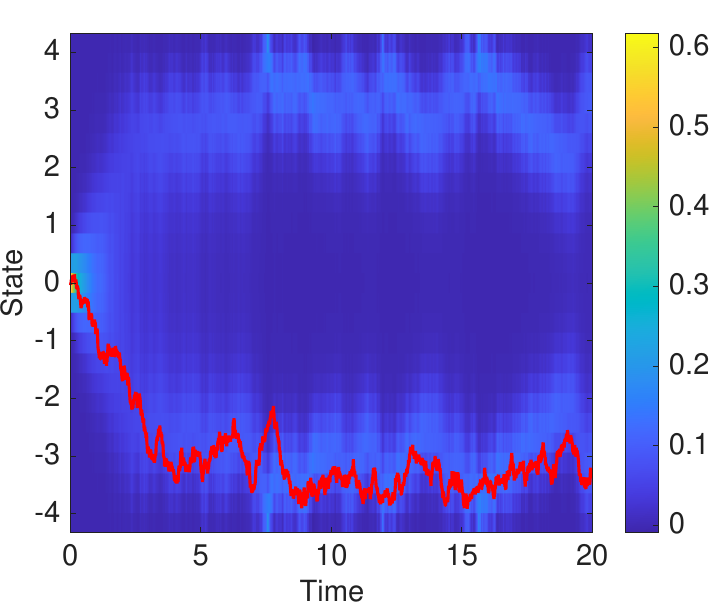}
		\caption{4th Dimension}
		\label{fig:tt_dim4}
	\end{subfigure}
	
	\caption{Tracking results of TT method for 4D multi-mode sensor problem \eqref{eq:multi_mode_prob}. Background heatmap shows the recovered marginal density function, and the red line is the true state trajectory.}
	\label{fig:tt_4d}
\end{figure}

\begin{figure}[htbp]
	\centering
	\begin{subfigure}{0.24\textwidth}
		\centering
		\includegraphics[width=\linewidth]{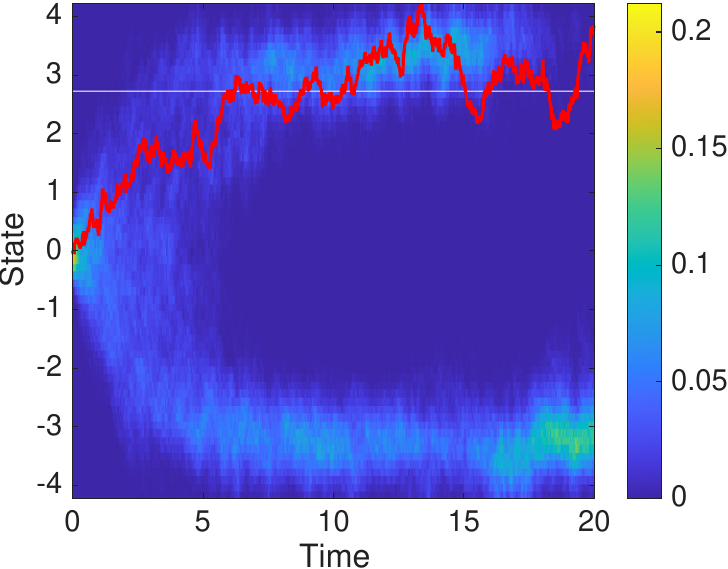}
		\caption{1st Dimension}
		\label{fig:pf_dim1}
	\end{subfigure}
	\hfill
	\begin{subfigure}{0.24\textwidth}
		\centering
		\includegraphics[width=\linewidth]{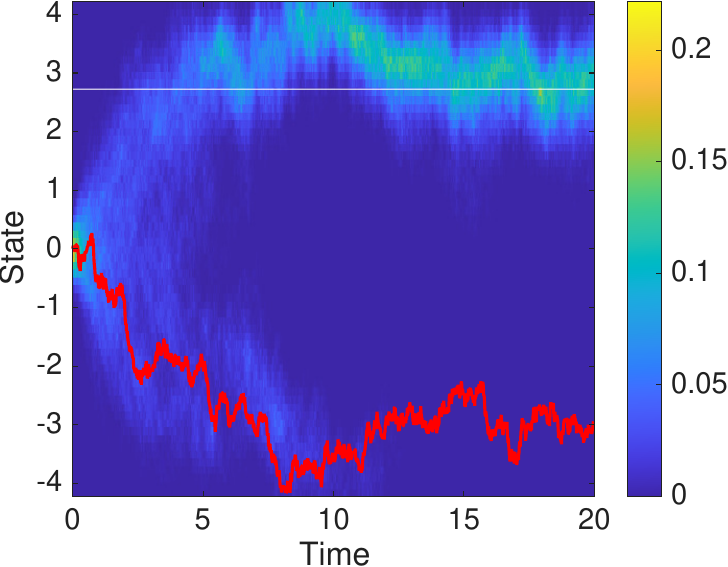}
		\caption{2nd Dimension}
		\label{fig:pf_dim2}
	\end{subfigure}
	\hfill
	\begin{subfigure}{0.24\textwidth}
		\centering
		\includegraphics[width=\linewidth]{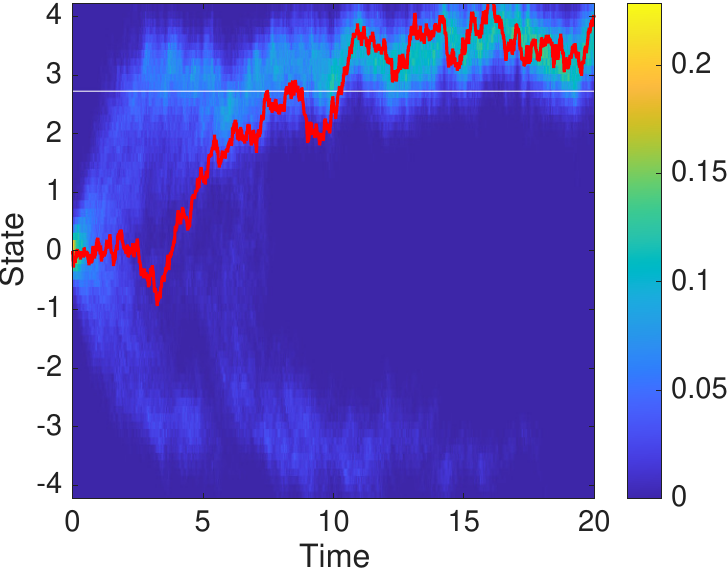}
		\caption{3rd Dimension}
		\label{fig:pf_dim3}
	\end{subfigure}
	\hfill
	\begin{subfigure}{0.24\textwidth}
		\centering
		\includegraphics[width=\linewidth]{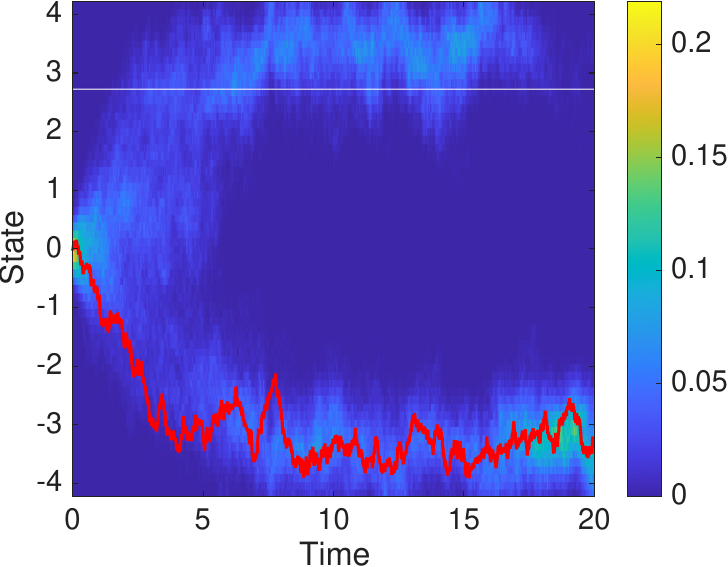}
		\caption{4th Dimension}
		\label{fig:pf_dim4}
	\end{subfigure}
	
	\caption{Tracking results of PF method (5000 particles) for 4D multi-mode sensor problem \eqref{eq:multi_mode_prob}. Background heatmap shows normalized histograms of the marginal densities, and the red line is the true state trajectory.}
	\label{fig:pf_4d}
\end{figure}

\begin{figure}[htbp]
	\centering
	\begin{subfigure}{0.24\textwidth}
		\centering
		\includegraphics[width=\linewidth]{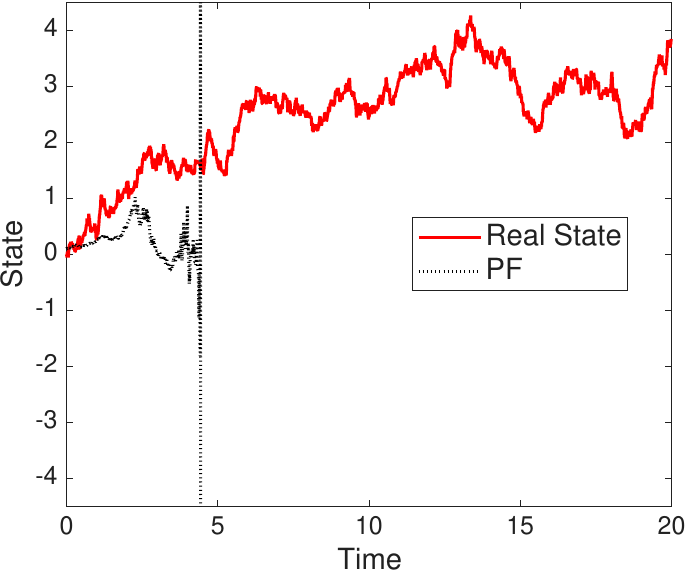}
		\caption{1st Dimension}
		\label{fig:ekf_dim1}
	\end{subfigure}
	\hfill
	\begin{subfigure}{0.24\textwidth}
		\centering
		\includegraphics[width=\linewidth]{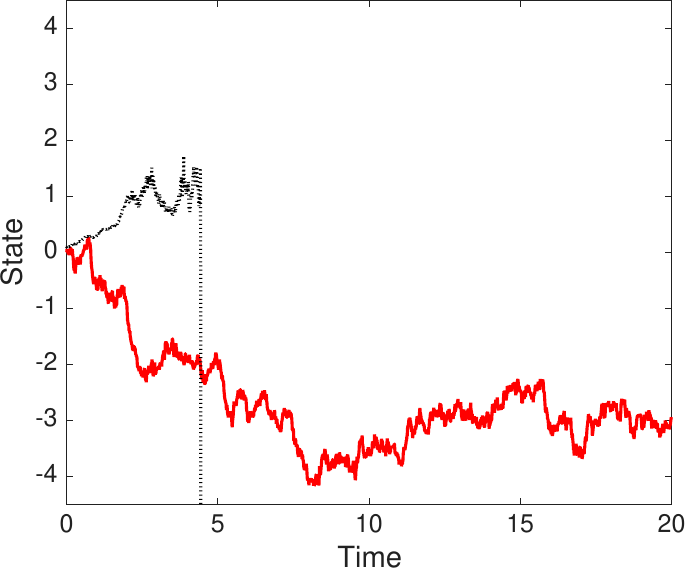}
		\caption{2nd Dimension}
		\label{fig:ekf_dim2}
	\end{subfigure}
	\hfill
	\begin{subfigure}{0.24\textwidth}
		\centering
		\includegraphics[width=\linewidth]{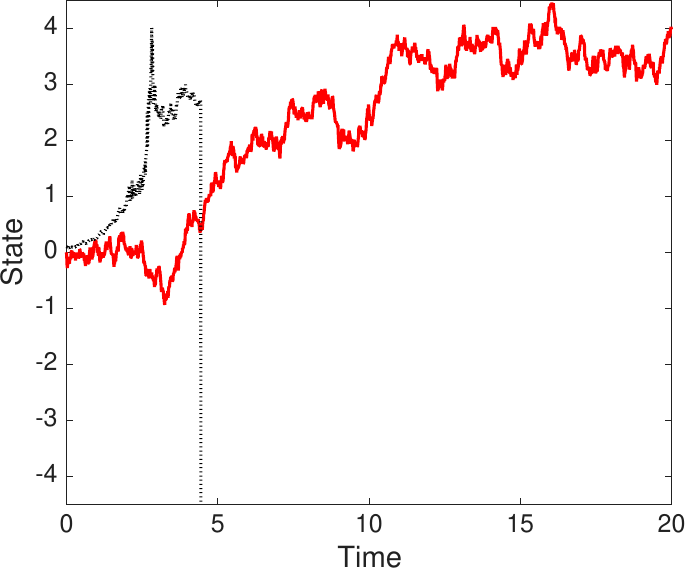}
		\caption{3rd Dimension}
		\label{fig:ekf_dim3}
	\end{subfigure}
	\hfill
	\begin{subfigure}{0.24\textwidth}
		\centering
		\includegraphics[width=\linewidth]{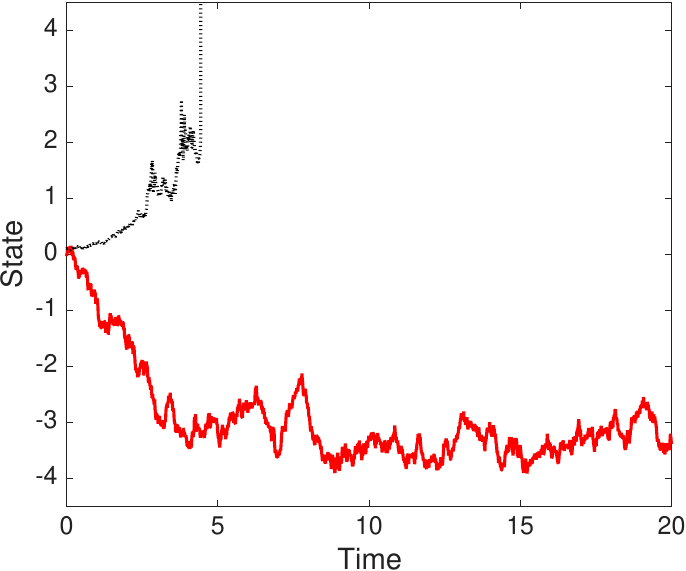}
		\caption{4th Dimension}
		\label{fig:ekf_dim4}
	\end{subfigure}
	
	\caption{Tracking results of EKF method for 4D multi-mode sensor problem \eqref{eq:multi_mode_prob}.}
	\label{fig:ekf_4d}
\end{figure}

\newpage
\section{Conclusions}
\label{sec:conclusions}
	In this paper, we develop a tensor train-based method for nonlinear filtering problems with correlated noise. By discretizing the DMZ equation, we obtain a high-dimensional SDE, which is efficiently represented in the TT format. Under suitable assumptions on the spatial regularity of the DMZ solution, as well as on the drift and observation functions in the signal process model \eqref{eq:signal_model}, we establish the spatial convergence of the TT semi-discrete system. 

For the temporal discretization, we employ a semi-implicit Milstein method and prove convergence by estimating related stochastic integrals, without invoking a change-of-measure argument. Numerical results demonstrate that the proposed method achieves accurate and stable performance in real time for problems up to 7D, and remains computationally tractable in higher dimensions, outperforming the particle filter and the extended Kalman filter.

In future work, we will extend the proposed method to more complex observation models and further improve its scalability for higher-dimensional and challenging nonlinear systems. 
In addition, we will develop dynamically low-dimensional approximation methods (see e.g.,\cite{cheng2013dynamically1,cheng2013dynamically2}) to solve NLF problems with time-varying drift and observation functions.

\appendix

\section{Discretization method for more general case}
If Assumption \eqref{ass:matrix-simplify} is generalized to the case in which ${G}{G}^\top$ and ${\rho}{\rho}^\top$ are not diagonal matrices, this gives rise to mixed second-order derivatives $\frac{\partial^2 u}{\partial x_i \partial x_j}$. In order to maintain the spatial convergence of $|| \mathcal{E} ||^2$ at a rate of $\mathcal{O} ((\Delta x)^4+\varepsilon^2)$, we should update the discretized operator ${\bf L}_0$ by
\begin{align}
	{\bf L}_0&=\frac{1}{2}({\bf \Delta}_{\rho}+{\bf \Delta}_{G}+{\bf M}^{\text{mix}} )-{\bf C}_d-{\bf M}^{\text{mix}}_{(0)},\notag
\end{align}
where 
\begin{align*}
	&{\bf M}^{\text{mix}} := \sum_{i\neq j} {\bf C}_d^{(i)} {\bf M}_{( GG^\top+\rho\rho^\top)_{ij}}  {\bf C}_d^{(j)}, \\
	&{\bf M}^{\text{mix}}_{(0)} := {\bf M}_{\operatorname{div}(f)} -\frac{\delta}{2}{\bf M}_{\nabla\cdot(\nabla\cdot (GG^\top + \rho\rho^\top)) }.
\end{align*}
As for time discretization, we can generalize the semi-implicit Milstein method \eqref{eq:implicit_Milstein} by 
\begin{align*}
	\hat{\bf U}_{n+1} =& \hat{\bf U}_n + \frac{1}{2}(\widetilde{\bf \Delta}_G+ \widetilde{\bf \Delta}_{\rho}  )\hat{\bf U}_{n+1} \delta\\
	& +(\frac{1}{2}  \widetilde{\bf M}^{\text{mix}}-\widetilde{\bf C}_d - \widetilde{\bf M}^{\text{mix}}_{(0)})\hat{\bf U}_n\delta\notag\\
	&+ \sum_{j=1}^d \widetilde{\bf L}_j \hat{\bf U}_nI_{(j)}^n +\sum_{i,j=1}^d\widetilde{\bf L}^{(i,j)}\hat{\bf U}_nI_{(i,j)}^n,
\end{align*}
In this case, we need the additional assumption that $G\in C^4_b$ as $\rho$ does. The proof of spatial and temporal convergence follows the same general framework as in Section \ref{sec:spatial convergence}. The main difference lies in the presence of mixed derivative terms, which introduce additional cross terms in the discrete energy estimates. These terms can be controlled using the symmetry and positive definiteness of the diffusion matrices $GG^\top$ and $\rho\rho^\top$, together with the second-order accuracy of the mixed difference operators ${\bf M}_{(0)}^{\text{mix}}$. 

\section{Proof of Lemma \eqref{lemma:integral_estimate}} \label{app:integral proof}

\textbf{Proof. } We shall prove it by induction. First, consider $\alpha$ with $l(\alpha)=1$. If $\alpha=(0)$, then
\begin{align}
	&\mathbb{E}_{\tau}(|I_{(0)}[g(\cdot)]_{\tau,t}|^2) = \mathbb{E}_{\tau} (|\int_{\tau}^t g(z) dz |^2)
	\notag\\ 
	\le& \mathbb{E}_\tau(\delta\int_{\tau}^t|g(z)|^2dz )\le \delta^2\mathcal{R}_{\tau,t}(g).
\end{align}
If $\alpha=(j)$, $j=1,...,d$, then
\begin{align}
	&\mathbb{E}_{\tau}(|I_{(k)}[g(\cdot)]_{\tau,t}|^2) \notag\\
	=& \mathbb{E}_{\tau} \Big(|\int_{\tau}^t g(z) h_k(x_z)dz + \int_{\tau}^t g(z) dw_z^k |^2 \Big)
	\notag\\ 
	\le& 2\mathbb{E}_\tau \Big(|\int_{\tau}^t g(z) h_k(x_z)dz|^2  +   |\int_{\tau}^t g(z) dw_z^k |^2 \Big)
	\notag\\
	\le &2 \mathbb{E}_{\tau} \Big( \delta \int_{\tau}^t  |g(z)h_k(x_z)|^2dz + \int_{\tau}^t s|g(z)|^2dz  \Big) 
	\notag\\
	\le & 4\mathbb{E}_{\tau} \Big(  \int_{\tau}^t  |g(z)|^2dz  \Big) 
	\le 4 \delta \mathcal{R}_{\tau,t}(g).
\end{align}

Now suppose that the estimate holds for $l(\alpha)=k$, consider $\alpha=(j_1,...,j_k,j_{k+1})$. If $j_{k+1}=0$, by the inductive assumption,
\begin{align}
	&\mathbb{E}_\tau (|I_{\alpha}[g(\cdot)_{\tau,t}]|^2) \notag\\
	=& \mathbb{E}_{\tau} \Big( \big| \int_{\tau}^t I_{\alpha-}[g(\cdot)]_{\tau,z} dz \big|^2 \Big)
	\le \mathbb{E}_{\tau} \Big( \delta \int_{\tau}^t |I_{\alpha-}[g(\cdot)]_{\tau,z} |^2 dz \Big)
	\notag\\
	\le& 4^{l(\alpha-)-n(\alpha-)}\delta \int_{\tau}^t (z-\tau)^{l(\alpha-)+n(\alpha-)} \mathcal{R}_{\tau,t}(g)
	\notag\\
	\le & 4^{l(\alpha)-n(\alpha)}\delta^{l(\alpha)+n(\alpha)} \mathcal{R}_{\tau,t}(g).
\end{align}

If $j_{k+1}=1,...,d$, then by It\^o isometry, $\delta<\frac{1}{C_h^2}$ and inductive assumption,
\begin{align}
	&\mathbb{E}_\tau (|I_{\alpha}[g(\cdot)_{\tau,t}]|^2) \notag\\
	=& \mathbb{E}_{\tau} \Big( \big| \int_{\tau}^t I_{\alpha-}[g(\cdot)]_{\tau,z} h_{k+1}(x_z)dz
	+ \int_{\tau}^t I_{\alpha-}[g(\cdot)]_{\tau,z} dw_z^{k+1} \big|^2 \Big)
	\notag\\
	\le & 2 \mathbb{E}_\tau \Big( \big|\int_{\tau}^t I_{\alpha-}[g(\cdot)]_{\tau,z} h_{k+1}(x_z)dz \big|^2 
	+ \big|\int_{\tau}^t I_{\alpha-}[g(\cdot)]_{\tau,z} dw_z^{k+1} \big|^2 \Big)
	\notag\\
	\le & 2\mathbb{E}_\tau \Big(C_h^2\delta  \int_{\tau}^t |I_{\alpha-}[g(\cdot)]_{\rho,z}|^2dz + \int_{\tau}^t |I_{\alpha-}[g(\cdot)]_{\tau,z}|^2 dz \Big)
	\notag\\
	\le & 4\int_{\tau}^t \mathbb{E}_\tau ( |I_{\alpha-}[g(\cdot)]_{\tau,z}|^2)dz \notag\\
	\le& 4^{l(\alpha)-n(\alpha)} \delta^{l(\alpha)+n(\alpha)}\sup_{\tau\le s\le t}\mathbb{E}_\tau(|g(s)|^2).
\end{align}
Then \eqref{eq:sto-int-square-est} follows. By Cauchy-Schwarz inequality,
\begin{align*}
	|\mathbb{E}_\tau(I_{\alpha;\tau,t})| \le \sqrt{\mathbb{E}_{\alpha;\tau,t}}\le 2^{l(\alpha)-n(\alpha)}\delta^{\frac{1}{2}(l(\alpha)+n(\alpha)) }.
\end{align*}
Substitute $g(\cdot)$ by ${\bf U}(\cdot)$, with Lemma \ref{lemma:integral_estimate}, we have $$\sup_{\tau\le s\le t}\mathbb{E}_\tau(|{\bf U}(\cdot)|^2)\le \sup_{\tau\le s\le t}e^{K_{\bf U}s}|{\bf U}(\tau)|^2\le e^{K_{\bf U}\delta}|{\bf U}(\tau)|^2,$$ the estimates follows:
\begin{align*}
	\mathbb{E}_\tau(| I_\alpha[{\bf U}(\cdot)]_{\tau,t} |^2) \le 4^{l(\alpha)-n(\alpha)} \delta^{l(\alpha)+n(\alpha)}e^{K_{\bf U}\delta}|{\bf U}(\tau)|^2.
\end{align*} 
\hfill$\square$

\section{Estimate in the proof of Lemma \ref{lemma:imMilstein_stability}}\label{app:RHS estimate}
In this section, we give a detailed derivation of \eqref{eq:RHS bound}. Recall that the term is given by
\begin{align*}
	&\mathbb{E}_{\tau_n}\big(|({\bf I}-{\bf A}_0\delta+\sum_k{\bf L}_kI_{(k)}^n +\sum_{k_1,k_2}{\bf L}^ {(k_1,k_2)} I_{(k_1,k_2)}^n) \hat{\bf U}_n|^2\big) \notag\\
	=& \underbrace{|({\bf I}-{\bf A}_0\delta)\hat{\bf U}_n|^2}_{:=\text{I}} 
	+\underbrace{ 2 \sum_k \mathbb{E}_{\tau_n}(I_{(k)}^n)\hat{\bf U}_n^\top({\bf I}-{\bf A}_0\delta)^\top{\bf L}_k\hat{\bf U}_n}_{:=\text{II}} \notag\\
	&+\underbrace{\sum_{k_1,k_2}\mathbb{E}_{\tau_n}(I_{(k_1)}^nI_{(k_2)}^n)\hat{\bf U}_n^\top{\bf L}_{k_1}^\top{\bf L}_{k_2}\hat{\bf U}_n}_{:=\text{III}}\notag\\
	&+\underbrace{ 2 \sum_{k_1,k_2}\mathbb{E}_{\tau_n}\big(I_{(k_1,k_2)}^n  \big)\hat{\bf U}_n^\top({\bf I}-\delta{\bf A}_0^\top) {\bf L}^ {(k_1,k_2)}   \hat{\bf U}_n}_{:=\text{IV}} \notag\\
	&
	+ \underbrace{2\sum_{k_0,k_1,k_2} \mathbb{E}_{\tau_n}\big(I_{(k_1,k_2)}^nI_{(k_0)}^n \big) \hat{\bf U}_n^\top{\bf L}_{k_0}^\top{\bf L}_ {k_2} {\bf L}_{k_1} \hat{\bf U}_n}_{:=\text{V}}\notag\\
	&+\underbrace{ \sum_{k_1,k_2,k_1',k_2'}\mathbb{E}_{\tau_n}\big( I_{(k'_1,k'_2)}^nI_{(k_1,k_2)}^n \big)
		\hat{\bf U}_n^\top {\bf L}_{k_1'}^\top {\bf L}_{k_2'}^\top {\bf L}_ {k_2} {\bf L}_{k_1} \hat{\bf U}_n }_{:=\text{VI}}.
\end{align*}
Before the estimate, we present some commonly used facts. For $\gamma\in\mathbb{R}$, ${\bf x}\in\mathbb{R}^n$, ${\bf A}\in\mathbb{R}^{n\times n}$,  $|\gamma {\bf x}^\top{\bf A}{\bf x}|\le|\gamma|\max(\mu({\bf A}),\mu(-{\bf A}))|{\bf x}|^2$. The logarithmic norm has the subadditivity property: $\mu({\bf A}+ {\bf B}) \le \mu(\bf A)+\mu({\bf B})$. If ${\bf A}$ is a diagonal matrix with $a$ being the bound of the entries, then $\mu( \pm ({\bf C}_d^{(i)})^\top {\bf A} {\bf C}_d^{(j)}) \le \frac{a}{(\Delta x)^2}$. 

In order to give a clean proof while indicating the dependence of the estimate on $\Delta x$, here we introduce the generic constant $c$, which is independent of $\delta$ and $\Delta x$.

Now we estimate the terms separately.
By the Lemma \eqref{lemma:lognorm-bnd}, and Assumption \eqref{ass:regularity}, \eqref{ass:growth}, $\text{I}$ is bounded by 
\begin{align}\label{eq:I_est}
	\text{I}=&
	|\hat{\bf U}_n|^2-2\delta\hat{\bf U}_n^\top{\bf A}_0\hat{\bf U}_n + \delta^2|{\bf A}_0\hat{\bf U}_n|^2\notag\\
	\le & |\hat{\bf U}_n|^2 + 2\delta \mu(-{\bf A}_0)|\hat{\bf U}_n|^2 + \delta^2(|{\bf C}_d\hat{\bf U}_n|^2+|{\bf M}_{(0)}\hat{\bf U}|^2)\notag\\
	%\le & \big(1+3\delta dL_f+\delta^2(d^2\frac{C_f^2}{(\Delta x)^2}+ d^2\frac{C_fL_f}{(\Delta x)}+d^2L_f^2 ) \big) |\hat{\bf U}_n|^2\notag\\
	\le & \Big(1+c\delta+ c\delta^2 \frac{1}{(\Delta x)^2} \Big) |\hat{\bf U}_n|^2%\notag\\
	%\le & \Big(1+\delta d\big(3L_f+\rho{C_f^2}+ \rho{C_fL_f} \big) + \delta^2d^2L_f^2\Big) |\hat{\bf U}_n|^2
\end{align}

\noindent For the term $\text{II}$, 
\begin{align}\label{eq:II_est}
	\text{II}
	\le& \sum_k|\mathbb{E}_{\tau_n}(I_{(k)}^n)|\cdot|\hat{\bf U}_n|^2 
	\cdot|\max\big(\mu({\bf L}_k-\delta{\bf A}_0^\top{\bf L}_k), \mu(-{\bf L}_k+\delta {\bf A}_0^\top{\bf L}_k)\big)|\notag\\
	%\le & dC_h\delta\Big( \frac{C_h}{s} + \delta d\big(\frac{1}{2}(L_fC_h+L_hC_f)+\rho\frac{C_f}{(\Delta x)^2} +  \frac{L_fC_h}{s}  + \frac{\rho L_f}{\Delta x}\big)\Big)|\hat{\bf U}_n|^2\notag\\
	\le &c(\delta  + \frac{\delta^2} {(\Delta x)^2} )\hat{\bf U}_n|^2.
	%\le &\delta dC_h\big(\frac{C_h}{s} +{\rho C_f} \big)|\hat{\bf U}_n|^2
	% +\delta^2C_hd^2\big(\frac{L_fC_h+L_hC_f}{2} +  \frac{L_fC_h}{s}+L_f \big)|\hat{\bf U}_n|^2
\end{align}

\noindent For the term $\text{III}$, we first estimate $|\mathbb{E}_{\tau_n}(I_{(k_1)}^nI_{(k_2)}^n)|$. 
\begin{align}
	&|\mathbb{E}_{\tau_n}(I_{(k_1)}^nI_{(k_2)}^n)|\notag\\
	=& \Big|\mathbb{E}_{\tau_n} \Big(  \int_{\tau_n}^{\tau_{n+1}} h_{k_1}(x_t)dt\cdot \int_{\tau_n} ^{\tau_{n+1}} w_t^{k_2} dt 
	+ \int_{\tau_n}^{\tau_{n+1}} h_{k_2}(x_t)dt\cdot \int_{\tau_n} ^{\tau_{n+1}} w_t^{k_1} dt \notag\\
	& + \int_{\tau_n}^{\tau_{n+1}} h_{k_1}(x_t)dt\cdot \int_{\tau_n} ^{\tau_{n+1}} h_{k_2}(x_t)dt
	+ \int_{\tau_n}^{\tau_{n+1}} w_t^{k_1} dt \cdot \int_{\tau_n} ^{\tau_{n+1}} w_t^{k_2} dt
	\Big) \Big| \notag\\
	\le & 	\Big( \mathbb{E}_{\tau_n}\big(\big|\int_{\tau_n}^{\tau_{n+1}} h_{k_1}(x_t)dt\big |^2 \big) \mathbb{E}_{\tau_n}\big(\big|\int_{\tau_n}^{\tau_{n+1}} w_t^{k_2}dt \big|^2 \big) \Big)^{1/2}\notag\\
	+&  \Big( \mathbb{E}_{\tau_n}\big(\big|\int_{\tau_n}^{\tau_{n+1}} h_{k_2}(x_t)dt\big |^2 \big) \mathbb{E}_{\tau_n}\big(\big|\int_{\tau_n}^{\tau_{n+1}} w_t^{k_1}dt \big|^2 \big) \Big)^{1/2} 
	+ C_h^2\delta^2 + \delta_{k_1,k_2}\delta\notag\\
	\le& C_h^2\delta^2 + 2C_h\delta^{3/2}+\delta_{k_1,k_2}\delta,
\end{align}
where $\delta_{ij}$ refers to the kronecker delta function. Then
\begin{align}\label{eq:III_est}
	\text{III}
	\le&\sum_{k_1\neq k_2}(2sC_h\delta^{3/2}+C_h^2\delta^2)|\hat{\bf U}_n|^2
	\cdot\max\big(\mu({\bf L}_{k_1}^\top{\bf L}_{k_2}),\mu(-{\bf L}_{k_1}^\top{\bf L}_{k_2})|\big)\notag\\
	&+ \sum_k(2C_h\delta^{3/2}+C_h^2\delta^2+\delta)||{\bf L}_k||_2^2 |\hat{\bf U}_n|^2\notag\\
	\le & c\frac{\delta}{(\Delta x)^2}|\hat{\bf U}_n|^2
	%\le & (\delta ds+\delta^{3/2}2sd^2C_h+\delta^2d^2C_h^2)(\frac{C_h^2}{s^2}+ \frac{2\rho L_h}{s}+1)|\hat{\bf U}_n|^2.
\end{align}

\noindent For the term $\text{IV}$, by the integral estimate in Lemma \ref{lemma:integral_estimate},
\begin{align}\label{eq:IV_est}
	\text{IV}
	& 	\le \sum_{k_1,k_2} \big| \mathbb{E}_{\tau_n}\big(I_{(k_1,k_2)}^n   \big) \big| |\hat{\bf U}_n|^2
	\cdot\max\Big(\mu\big(({\bf I}-\delta{\bf A}_0^\top){\bf L}_ {k_2} {\bf L}_{k_1} \big) , \mu\big(-({\bf I}-\delta{\bf A}_0^\top){\bf L}_ {k_2} {\bf L}_{k_1} \big) \Big)\notag\\
	\le &c\frac{\delta}{(\Delta x )^2} |\hat{\bf U}_n|^2
\end{align}

\noindent For the term $\text{V}$ and $\text{VI}$,  
\begin{align}\label{eq:V_est}
	\text{V}
	\le & \sum_{k_0,k_1,k_2}|\mathbb{E}_{\tau_n}\big(I_{(k_1,k_2)}^n \Delta I_{(k_0)}^n  \big)|\cdot||{\bf L}_{k_0}^\top{\bf L}_{k_1}{\bf L}_{k_2}||_2|\hat{\bf U}_n|^2\notag\\
	\le & c\frac{\delta^{3/2}}{(\Delta x)^{3} }|\hat{\bf U}_n|^2%\le d^3 (4s\delta )^{3/2}(\frac{C_h}{s}+1)^{3} |\hat{\bf U}_n|^2,
\end{align}
and
\begin{align}\label{eq:VI_est}
	\text{VI}
	\le &\sum_{k_1,k_2,k_1',k_2'} |\mathbb{E}_{\tau_n}\big( I_{(k'_1,k'_2)}^nI_{(k_1,k_2)}^n\big)| 
	\cdot|| ({\bf L}^{(k_1',k_2')})^\top {\bf L}^ {(k_1,k_2)} ||_2 |\hat{\bf U}_n|^2\notag\\
	\le & c\frac{\delta^2}{(\Delta x)^4} |\hat{\bf U}_n|^2
	%\le 16\delta^2s^2d^4(\frac{C_h}{s}+1)^{4} |\hat{\bf U}_n|^2
\end{align}

Combining \eqref{eq:I_est}, \eqref{eq:II_est}, \eqref{eq:III_est}, \eqref{eq:IV_est}, \eqref{eq:V_est} and \eqref{eq:VI_est}, the estimate is obtained as
\begin{align}
	&\mathbb{E}_{\tau_n}\big(|({\bf I}-{\bf A}_0\delta+\sum_k{\bf L}_kI_{(k)}^n +\sum_{k_1,k_2}{\bf L}_ {k_2} {\bf L}_{k_1}I_{(k_1,k_2)}^n) \hat{\bf U}_n|^2\big)\notag\\
	\le &(1+K_{\text{stab}}\delta)|\hat{\bf U}_n|^2,
\end{align}
where $K_{\text{stab}}$ is independent of $\delta$. 
\hfill$\square$

\newpage
\printbibliography
%\bibliographystyle{plain}   % or abbrv, siam, etc.
%\bibliography{references}  
\end{document}